\documentclass[11pt]{amsart}

\usepackage{amsmath}

\usepackage{amssymb}

\usepackage{amsfonts}

\usepackage{amsthm}

\usepackage{enumerate}

 \usepackage{esint}

\usepackage{graphicx}

\usepackage{verbatim}

\newtheorem{innercustomthm}{Theorem}
\newenvironment{customthm}[1]
  {\renewcommand\theinnercustomthm{#1}\innercustomthm}
  {\endinnercustomthm}

\newcommand{\Norm}[1]{ \left\|  #1 \right\| }

\newcommand{\floor}[1]{\lfloor #1 \rfloor }

\newcommand{\Be}{\begin{equation}}

\newcommand{\Ee}{\end{equation}}

\newcommand{\Bm}{\begin{multline}}

\newcommand{\Em}{\end{multline}}

\newcommand{\Bea}{\begin{eqnarray}}

\newcommand{\Eea}{\end{eqnarray}}

\newcommand{\Beas}{\begin{eqnarray*}}

\newcommand{\Eeas}{\end{eqnarray*}}

\newcommand{\Benu}{\begin{enumerate}}

\newcommand{\Eenu}{\end{enumerate}}

\newcommand{\Bi}{\begin{itemize}}

\newcommand{\Ei}{\end{itemize}}

\def\intslash{\fint}

\def\intslash{\rlap{\kern  .32em $\mspace {.5mu}\backslash$ }\int}

\def\qsl{{\rlap{\kern  .32em $\mspace {.5mu}\backslash$ }\int_{Q_x}}}

\def\Norm#1{{ \left\|  #1 \right\| }}

\def\floor#1{{\lfloor #1 \rfloor }}

\def\emph#1{{\it #1 }}

 at 10 true pt

\def\be#1{\begin{equation}\label{ #1}}

\def\endeq{\end{equation}}

\def\endal{\end{align}}

\def\bas{\begin{align*}}

\def\eas{\end{align*}}

\def\bi{\begin{itemize}}

\def\ei{\end{itemize}}

\def\emph#1{{\it #1}}

\def\textbf#1{{\bf #1}}

\theoremstyle{plain}

   \newtheorem{theorem}{Theorem}[section]

   \newtheorem{proposition}[theorem]{Proposition}

   \newtheorem{lemma}[theorem]{Lemma}

   \newtheorem{corollary}[theorem]{Corollary}

   \newtheorem{theorem*}{Theorem}

\theoremstyle{remark}

\theoremstyle{definition}

\numberwithin{equation}{section}

\begin{document}
\title[On a generalized Bochner-Riesz square function]{On the square function associated with generalized Bochner-Riesz means}

\author[L. Cladek]{Laura Cladek}

\address{L. Cladek, Department of Mathematics\\ University of Wisconsin-Madison\\480 Lincoln Drive, Madison, WI 53706, USA}

\email{cladek@math.wisc.edu}

\subjclass{42B15}

\begin{abstract}
We consider generalized Bochner-Riesz multipliers of the form $(1-\rho(\xi))_+^{\lambda}$ where $\rho:\mathbb{R}^2\to\mathbb{R}$ belongs to a class of rough distance functions homogeneous with respect to a nonisotropic dilation group. We prove a critical $L^4$ estimate for the associated square function, which we use to derive multiplier theorems for multipliers of the form $m\circ\rho$ where $m:\mathbb{R}\to\mathbb{C}$.
\end{abstract}

\thanks{The author would like to thank Andreas Seeger for introducing this problem, and for his guidance and many helpful discussions.}
\maketitle

\section{Introduction}

The characterization of Fourier multiplier operators that are bounded on $L^p$ when $p\ne 1, 2$ is a difficult open problem that has a long and rich history in harmonic analysis. A particular special case that has been especially studied is the class of radial Fourier multipliers, for which the Bochner-Riesz multipliers are prototypical examples. In \cite{cgt}, Carbery, Gasper and Trebels proved sufficient conditions for a radial function on $\mathbb{R}^2$ to be a Fourier multiplier on $L^p(\mathbb{R}^2)$. Their theorem can be stated as follows.

\begin{customthm}{A}[\cite{cgt}]\label{cgtthm}
Let $m: (0, \infty)\to\mathbb{C}$ be bounded and measurable. Then for $4/3\le p\le 4$ and $\alpha>1/2$,
\begin{align*}
\Norm{m(|\cdot|)}_{M^p(\mathbb{R}^2)}\lesssim\sup_{t>0}\bigg(\int|\mathcal{F}_{\mathbb{R}}[\phi(\cdot)m(t\cdot)](\tau)|^2|\tau|^{2\alpha}\,d\tau\bigg)^{1/2}.
\end{align*}
\end{customthm}
Theorem \ref{cgtthm} is sharp, as can be verified by comparing with the known sharp $L^p$ bounds for Bochner-Riesz multipliers in $\mathbb{R}^2$ (see \cite{feff}). Theorem \ref{cgtthm}  was obtained as a consequence of a critical $L^4$ estimate for the Bochner-Riesz square function in $\mathbb{R}^2$, proved by Carbery in \cite{car2}.
\newline
\indent
In this paper, we extend the result of Theorem \ref{cgtthm} to a class of \textit{quasiradial} multipliers of the form $m\circ\rho$, where $\rho$ belongs to a class of rough distance functions homogeneous with respect to a \textit{nonisotropic} dilation group. Here we may view $\rho(\xi)$ as generalizing the function $|\xi|$, which corresponds to the special case of radial multipliers. Our consideration of such a class of distance functions is in part motivated by the work of Seeger and Ziesler in \cite{sz}, where the authors consider Bochner-Riesz means of the form $(1-\rho(\xi))_+^{\lambda}$ where $\rho$ is the Minkowski functional of a bounded convex domain in $\mathbb{R}^2$ containing the origin. However, the class of distance functions we work with is more general than what is considered in \cite{sz}, since it also includes distance functions $\rho$ that have nonisotropic homogeneity.
\newline
\indent
As motivated by \cite{sz}, let $\Omega\subset\mathbb{R}^2$ be a bounded, open convex set containing the origin. Since the results in this paper are dilation invariant, we will assume that $\Omega$ contains the ball of radius $8$ centered at the origin. Let $M>0$ be the smallest positive integer such that
\begin{align}\label{quantM}
\{\xi:\,|\xi|\le 8\}\subset\Omega\subset\overline{\Omega}\subset\{\xi:\,|\xi|<2^M\}.
\end{align}
This quantity $M$ associated to such a convex domain $\Omega$ is an important parameter on which our results will depend. One may note that it determines the Lipschitz norm of parametrizations of $\partial\Omega$. 
\newline
\indent
We now introduce the notion of a nonisotropic dilation group. Let $A$ be a $2\times 2$ matrix with eigenvalues $\lambda_1$, $\lambda_2$ (not necessarily distinct) such that $\text{Re}(\lambda_1), \text{Re}(\lambda_2)>0$. A \textit{nonisotropic dilation group} associated to $A$ is a one-parameter family $\{t^A:\,t>0\}$, where $t^A=\exp(\log(t)A)$. We say that a pair $(\Omega, A)$ is \textit{compatible} if it satisfies the following:
\begin{enumerate}
\item For any $\xi\in\mathbb{R}^2\setminus\{0\}$ the orbit $\{t^A\xi:\,t>0\}$ intersects $\partial\Omega$ exactly once,
\item If $\Theta(\Omega, A)$ denotes the infimum of all angles between the tangent vector to an orbit $\{t^A\xi:\,t>0\}$ at $\xi$ and a supporting line at $\xi$ for any $\xi\in\partial\Omega$, then $\Theta(\Omega, A)>0$.
\end{enumerate}
We associate to a compatible pair $(\Omega, A)$ a norm function $\rho\in C(\mathbb{R}^2)$, defined by setting $\rho(0)=0$ and setting $\rho(\xi)$ to be the unique $t$ such that $t^{-A}\xi\in\partial\Omega$ if $\xi\ne 0$. If $\partial\Omega$ is smooth, then $\rho\in C^{\infty}(\mathbb{R}^2\setminus\{0\})$. To see this, apply the implicit function theorem to $F(x, t)=\text{dist}(t^Ax, \partial\Omega)$. Moreover, we also have $\Norm{\rho}_{C^{0, 1}(K)}\lesssim_{K, M, \text{Re}(\lambda_1), \text{Re}(\lambda_2), \Theta(\Omega, A)}1$ for any compact $K\subset\mathbb{R}^2\setminus\{0\}$.
\newline
\indent
Note that in the special case that $A$ is the identity, $(\Omega, A)$ is a compatible pair for any bounded, open convex set $\Omega$ satisfying (\ref{quantM}), and we have $\Theta(\Omega, A)\gtrsim_M1$. It was noted in \cite{sw} that for every $A$ there exists a compatible pair $(\Omega, A)$ obtained by taking $\Omega$ to be the region bounded by $\{\xi\in\mathbb{R}^2: \left<B\xi, \xi\right>=1\}$, where $B$ is the positive definite symmetric matrix given by
\begin{align*}
B=\int_0^{\infty}\exp(-tA^{\ast})\exp(-tA)\,dt.
\end{align*}
See \cite{sw} for more details. In this particular case $\partial\Omega$ is smooth; however as already noted in this paper we consider general convex domains, with special emphasis on the case when $\partial\Omega$ is rough.

\subsection*{Notation}
Throughout the rest of the paper, in every situation where it is clear that we have fixed a compatible pair $(\Omega, A)$, we will write $\lesssim$, $\gtrsim$ and $\approx$ to denote inequalities where the implied constant possibly depends on $M$, $\text{Re}(\lambda_1)$, $\text{Re}(\lambda_2)$, and $\Theta(\Omega, A)$. We will also assume that all explicit constants that appear possibly depend on $M$, $\text{Re}(\lambda_1)$, $\text{Re}(\lambda_2)$, and $\Theta(\Omega, A)$.
\newline
\newline
\indent
Given a compatible pair $(\Omega, A)$, define the Bochner-Riesz means $R_t^{\lambda}f$ associated with $(\Omega, A)$ for Schwartz functions $f\in\mathcal{S}(\mathbb{R}^2)$ by
\begin{align*}
\mathcal{R}_t^{\lambda}f(x)=\frac{1}{(2\pi)^2}\int_{|\xi|\le t}\bigg(1-\frac{\rho(\xi)}{t}\bigg)^{\lambda}\hat{f}(\xi)e^{i\left<\xi, x\right>}\,d\xi.
\end{align*}
Define the Bochner-Riesz square function $G^{\lambda}f$ associated with $(\Omega, A)$ for Schwartz functions $f\in\mathcal{S}(\mathbb{R}^2)$ by
\begin{align*}
G^{\lambda}f(x)=\bigg(\int_0^{\infty}\big|\mathcal{R}_t^{\lambda}f(x)\big|^2\,\frac{dt}{t}\bigg)^{1/2}.
\end{align*}
Our main result is the following critical $L^4$ estimate for the Bochner-Riesz square function.
\begin{theorem}\label{mainthm}
Let $(\Omega, A)$ be a compatible pair, and let $G^{\lambda}f$ denote the Bochner-Riesz square function associated to $(\Omega, A)$. For $\lambda>-1/2$, 
\begin{align*}
\Norm{G^{\lambda}f}_{L^4(\mathbb{R}^2)}\lesssim\Norm{f}_{L^4(\mathbb{R})^2}
\end{align*}
for $f\in\mathcal{S}(\mathbb{R}^2)$.
\end{theorem}
Following \cite{cgt}, we obtain the subsequent corollary, which is an extension of the result of Theorem \ref{cgtthm} to quasiradial multipliers of the form $m\circ\rho$.
\begin{corollary}\label{maincor}
Let $(\Omega, A)$ be a compatible pair with associated norm function $\rho$. Let $m:\mathbb{R}\to\mathbb{C}$ be measurable function with $\Norm{m}_{L^{\infty}(\mathbb{R})}\le 1$. Then for $4/3\le p\le 4$ and $\alpha>1/2$,
\begin{align*}
\Norm{m\circ\rho}_{M^p(\mathbb{R}^2)}\lesssim
\sup_{t>0}\bigg(\int|\mathcal{F}_{\mathbb{R}}[\phi(\cdot)m(t\cdot)](\tau)|^2|\tau|^{2\alpha}\,d\tau\bigg)^{1/2}.
\end{align*}
\end{corollary}

To prove Theorem \ref{mainthm}, we will first decompose the multiplier $(1-\rho(\xi))_+^{\lambda}$ in a standard fashion into smooth functions supported on ``annuli" of thickness comparable to the distance from $\partial\Omega$ (for example, see \cite{cor1}, \cite{car2}). Theorem \ref{mainthm} then reduces to proving the following proposition.
\begin{proposition}\label{mainprop}
Let $(\Omega, A)$ be a compatible pair. Fix a Schwartz function $\Phi:\mathbb{R}\to\mathbb{R}$ supported in $[-1, 1]$ with $|\Phi|\le 1$. There is a constant $C>0$ such that for every $\epsilon>0$ and every $0<\delta<C$,
\begin{align*}
\Norm{\bigg(\int_0^{\infty}|\psi_t\ast f(x)|^2\frac{\,dt}{t}\bigg)^{1/2}}_4\lesssim_{\epsilon}\delta^{1/2-\epsilon}\Norm{f}_4,
\end{align*}
where 
\begin{align*}
\psi_t(x)=\mathcal{F}(\phi(\frac{\rho(\cdot)}{t}))(x),\qquad \phi(\xi)=\Phi(\frac{\xi-1}{\delta}).
\end{align*}
\end{proposition}

The overall structure of the proof of Proposition \ref{mainprop} will follow \cite{car2} and \cite{sz}, and will draw heavily on the techniques therein. However, the presence of nonisotropic dilations and the roughness of $\partial\Omega$ introduces new difficulties to the proof since the underlying geometry becomes more complicated, requiring more intricate decompositions on the Fourier side as well as a more sophisticated use of Littlewood-Paley inequalities.

\section{preliminaries on convex domains in $\mathbb{R}^2$}

\subsubsection*{Elementary facts about convex functions in $\mathbb{R}^2$}
We note for later use the following lemma, which can be found in \cite{sz}. The proof is straightforward and we omit it here, and the reader is encouraged to refer to \cite{sz} for a proof.
\begin{lemma}[\cite{sz}]\label{elemlemma}
$\partial\Omega\cap\{x:\,-1\le x_1\le 1\}$ can be parametrized by
\begin{align}
t\mapsto (t, \gamma(t)),\qquad{}-1\le t\le 1,
\end{align}
where
\begin{enumerate}
\item
\begin{align}
1<\gamma(t)<2^M,\qquad{} -1\le t\le 1.
\end{align}
\item $\gamma$ is a convex function on $[-1, 1]$, so that the left and right derivatives $\gamma_L^{\prime}$ and $\gamma_R^{\prime}$ exist everywhere in $(-1, 1)$ and
\begin{align}
-2^{M-1}\le\gamma_R^{\prime}(t)\le\gamma_L^{\prime}(t)\le 2^{M-1}
\end{align}
for $t\in [-1, 1]$. The functions $\gamma_L^{\prime}$ and $\gamma_R^{\prime}$ are decreasing functions; $\gamma_L^{\prime}$ and $\gamma_R^{\prime}$ are right continuous in $[-1, 1]$.
\item Let $\ell$ be a supporting line through $\xi\in\partial\Omega$ and let $n$ be an outward normal vector. Then
\begin{align}
|\left<\xi, n\right>|\ge 2^{-M}|\xi|.
\end{align}
\end{enumerate}
\end{lemma}

\subsubsection*{Reduction to the case when $\partial\Omega$ is smooth.}

Motivated by \cite{sz}, Lemma $2.2$, we will show that it suffices to prove Proposition \ref{mainprop} with the implied constant depending only on $M$ (and not, for instance, the $C^2$ norm of local parametrizations of $\partial\Omega$) in the special case that $\partial\Omega$ is smooth. The first step is to approximate $\Omega$ by a sequence of convex domains with smooth boundaries satisfying the same quantitative estimates as $\Omega$.
\begin{lemma}\label{approxlemma}
Let $(\Omega, A)$ be a  compatible pair. There is a sequence of convex domains $\{\Omega_n\}$ satisfying the following:
\begin{enumerate}
\item $\partial\Omega_n$ is $C^{\infty}$,
\item For $n$ sufficiently large, $(\Omega_n, A)$ is a compatible pair and $\Theta(\Omega_n, A)\ge \Theta(\Omega, A)/2$,
\item For each $n$ we have \begin{align*}
\{\xi:\,|\xi|\le 4\}\subset\Omega_n\subset\overline{\Omega_n}\subset\{\xi:\,|\xi|<2^{M+1}\},
\end{align*}
\item $\lim_{n\to\infty}\rho_n(\xi)=\rho(\xi)$ with uniform convergence on compact sets.
\end{enumerate}
\end{lemma}

\begin{proof}
We adopt the same approach as in \cite{sz}, namely, approximating $\Omega$ by convex polygons and smoothing out the vertices. For each $n$, let $P_n$ be the polygon with vertices $\{v_1, \ldots, v_n\}$, where $v_i$ is the unique point on $\partial\Omega$ making an angle of $2\pi i/n$ with the $\xi_2$-axis. Then $P_n$ is convex and $P_n\subset\Omega$. Choose intervals $I_n=[x_{n, 0}, x_{n, 1}]\subset\tilde{I}_n=(\tilde{x}_{n, 0}, \tilde{x}_{n, 1})\subset\mathbb{R}$ centered at $0$ such that $\partial P_n\cap\{(\xi_1, \xi_2):\,\xi_1\in I_n, \xi_2>0\}$ can be parametrized as $\{(\alpha, \tilde{\gamma}_n(\alpha)):\,\alpha\in I_n\}$, and also so that $\{(\alpha, \tilde{\gamma}_n(\alpha):\,\alpha\in\tilde{I}_n\}$ does not contain any vertices of $P_n$ except $v_1$. 
\\
\indent
Now let $\eta\in C_0^{\infty}(\mathbb{R})$ be an even nonnegative function supported in $(-1/2, 1/2)$ so that $\int\eta(t)\,dt=1$. Let $C_n=100\max\{(x_{n, 0}-\tilde{x}_{n, 0})^{-1}, (\tilde{x}_{n, 1}-x_{n, 1})^{-1}\}$, and set
\begin{align*}
\gamma_n(\alpha)=\int C_n\,\eta(C_nt)\tilde{\gamma}_n(\alpha-t)\,dt,\qquad\alpha\in I_n.
\end{align*}
By the choice of $C_n$, we have that $\{(\alpha, \gamma_n(\alpha)):\,\alpha\in I_n\}$ coincides with $P_n$ near the endpoints of $I_n$. We may thus obtain a smooth convex curve $\partial\Omega_n$ by replacing $\partial P_n$ near $v_1$ with $\{(\alpha, \gamma_n(\alpha)):\,\alpha\in I_n\}$, and then repeating the same procedure near the other vertices $v_2, \ldots, v_n$ after performing appropriate rotations.
\newline
\indent
It is clear that $\{\Omega_n\}$ satisfies $(1), (3)$, and $(4)$, so it remains to show $(2)$. Let $\epsilon_0>0$ be sufficiently small so that for any $\xi\in\partial\Omega$ and $s_1, s_2\in [1-\epsilon_0, 1+\epsilon_0]$, the difference in slope between the tangent lines to the orbit $\{t^A\xi:\,t>0\}$ at $s_1^A\xi$ and the tangent line at $s_2^A\xi$ is less than $\Theta(\Omega, A)/10$. Now choose $0<\epsilon_1<\epsilon_0$ sufficiently small so that 
\begin{multline}\label{intersect}
\{t^A\xi:\,t>0, t\notin[1-\epsilon_0,1+\epsilon_0], \xi\in\partial\Omega\}
\\
\cap\{t^A\xi:\,t\in[1-\epsilon_1, 1+\epsilon_1], \xi\in\partial\Omega\}=\emptyset.
\end{multline}
Next, choose $N>0$ sufficiently large so that whenever $n\ge N$, the following holds:
\begin{enumerate}
\item $\partial\Omega_n\subset\{\xi:\,1-\epsilon_1\le\rho(\xi)\le 1+\epsilon_1\}$,
\item The difference in slope between the tangent line at any point $x\in\partial\Omega_n$ and some supporting line of $\partial\Omega$ at the vertex of $P_n$ nearest to $x$ is less than $\Theta(\Omega, A)/10$,
\item For any $\xi\in\partial\Omega_n$, the difference in slope between the tangent vector to the orbit $\{t^A\xi:\,t>0\}$ at $\xi$ and the tangent vector to the orbit $\{t^Av_i\}$ at $v_i$, where $v_i$ is the vertex of $P_n$ nearest to $\xi$, is less than $\Theta(\Omega, A)/10$.
\end{enumerate}
To see that we may choose $N$ so that $(1)$ and $(3)$ are satisfied is fairly obvious, and to see that we may choose $N$ so that $(2)$ holds requires only a straightforward application of (2) from Lemma \ref{elemlemma}. It is easy to see that $(2)$ and $(3)$ imply that $\Theta(\Omega_n, A)>\Theta(\Omega, A)/2$. $(1)$ and (\ref{intersect}) imply that $\{t^A\xi:\,t>0, t\notin[1-\epsilon_0,1+\epsilon_0], \xi\in\partial\Omega\}$ does not intersect $\partial\Omega_n$. Given $\xi\in\partial\Omega$, let $t(\xi)>0$ be the smallest value of $t$ such that $t^{-A}\xi\in\partial\Omega_n$. Then $t(\xi)\in [1-\epsilon_0, 1+\epsilon_0]$. But by the choice of $\epsilon_0$, any tangent line to $\{t^A\xi: t\in [1-\epsilon_0, 1+\epsilon_0]\}$ makes an angle of at least $\Theta(\Omega, A)/4$ with the tangent line to $\partial\Omega_n$ at $t^{-A}\xi$, and by convexity of $\partial\Omega_n$ there can be no $t>t(\xi)$ such that $t^{-A}\xi\in\partial\Omega_n$. Thus $(\Omega_n, A)$ is a compatible pair for $n\ge N$.

\end{proof}

\begin{lemma}\label{intlemma}
Suppose that Proposition \ref{mainprop} holds in the special case when $\partial\Omega$ is smooth, with a constant depending only on $M$, $\epsilon$, $\text{Re}(\lambda_1), \text{Re}(\lambda_2)$, and $\Theta(\Omega, A)$. Then Proposition \ref{mainprop} holds in the full stated generality.
\end{lemma}

\begin{proof}[Proof of Lemma \ref{intlemma}]
Let $\{\Omega_n\}$ be a sequence of convex domains approximating $\Omega$ as in Lemma \ref{approxlemma}, and suppose the statement of Proposition \ref{mainprop} holds in the special case of convex domains with smooth boundaries, with a constant depending only on the quantities listed in Lemma \ref{intlemma}. Fix a Schwartz function $f\in\mathcal{S}(\mathbb{R}^2)$. Then for every $\epsilon>0$ and every $0<\delta<C$, for $n$ sufficiently large we have
\begin{align*}
\Norm{\bigg(\int_0^{\infty}|\psi_{n, t}\ast f(x)|^2\frac{\,dt}{t}\bigg)^{1/2}}_4\le C_{\epsilon, M, \text{Re}(\lambda_1), \text{Re}(\lambda_2), \Theta(\Omega, A)}\delta^{1/2-\epsilon}\Norm{f}_4,
\end{align*}
where 
\begin{align*}
\psi_{n, t}(x)=\mathcal{F}(\phi(\frac{\rho_n(\cdot)}{t}))(x),\qquad \phi(\xi)=\Phi(\frac{\xi-1}{\delta}).
\end{align*}
Since $\phi(\frac{\rho_n(\cdot)}{t})\to\phi(\frac{\rho(\cdot)}{t})$ uniformly as $n\to\infty$, we have that $\psi_{n, t}\ast f(x)\to\psi_t\ast f(x)$ pointwise as $n\to\infty$. By Fatou's lemma applied twice,
\begin{align*}
\Norm{\bigg(\int_0^{\infty}|\psi_{t}\ast f(x)|^2\frac{\,dt}{t}\bigg)^{1/2}}_4\le\liminf_n\Norm{\bigg(\int_0^{\infty}|\psi_{n, t}\ast f(x)|^2\frac{\,dt}{t}\bigg)^{1/2}}_4
\\
\le C_{\epsilon, M, \text{Re}(\lambda_1), \text{Re}(\lambda_2), \Theta(\Omega, A)}\delta^{1/2-\epsilon}\Norm{f}_4,
\end{align*}
as desired.
\end{proof}

\section{An $L^2$ maximal function estimate}
In \cite{cor2}, C\'{o}rdoba proved $L^2$ bounds for the Nikodym maximal function in $\mathbb{R}^2$. These bounds were an important ingredient in \cite{car2} to prove Proposition \ref{mainprop} for the special case of the classical (radial) Bochner-Riesz means. To prove Proposition \ref{mainprop} in the full stated generality, we need a nonisotropic version of C\'{o}rdoba's result. To this end, we will closely follow \cite{cor2} to prove the following proposition.

\begin{proposition}\label{maxprop} Let $N, \lambda>0$ be real numbers, and let $\mathcal{C}$ be the collection of all rectangles in $\mathbb{R}^2$ with dimensions $\lambda$ and $N\lambda$. Let 
\begin{align*}
{\mathcal{C}_k}=\{(2^k)^AR:\,R\in\mathcal{C}, k\in\mathbb{Z}\}.
\end{align*}
Define a maximal operator $M_{\lambda, N}$ by
\begin{align*}{M}_{\lambda, N}f(x)=\sup_{x\in R\in\bigcup_k\mathcal{C}_k}\frac{1}{|R|}\int_R |f(y)|\,dy.
\end{align*} 
Then there is a constant $\beta(\text{Re}(\lambda_1), \text{Re}(\lambda_2))>0$ such that for every Schwartz function $f\in \mathcal{S}(\mathbb{R}^2)$, 
\begin{align*}
\Norm{{M}_{\lambda, N}f}_2\lesssim_{\text{Re}(\lambda_1), \text{Re}(\lambda_2)}\log(N)^{\beta(\text{Re}(\lambda_1), \text{Re}(\lambda_2))}\Norm{f}_2.
\end{align*}

\end{proposition}

\begin{proof}
In what follows, for any rectangle $R$ and any integer $k$, we will let $((2^k)^AR)^{\ast}:=(2^k)^A(R^{\ast})$. Here $R^{\ast}$ denotes the double dilate of $R$, where the dilation is taken from the center of $R$. Similarly, if $\mathcal{R}$ denotes any collection of nonisotropic dilates of rectangles, then $\mathcal{R}^{\ast}:=\{R^{\ast}:\,R\in\mathcal{R}\}$.
\newline
\indent For each $k\in\mathbb{Z}$, define a maximal operator $M_{\lambda, N, k}$ on Schwartz functions $f\in\mathcal{S}(\mathbb{R}^2)$ by
\begin{align}\label{kmax}
M_{\lambda, N, k}f(x)=\sup_{x\in R\in{\mathcal{C}_k}}\frac{1}{|R|}\int_R |f(y)|\,dy.
\end{align}
It follows from rescaling the corresponding result from \cite{cor2} that for every $f\in\mathcal{S}(\mathbb{R}^2)$,
\begin{align}\label{kmax1}
\Norm{M_{\lambda, N, k}f}_2\lesssim\log(3N)^{1/2}\Norm{f}_2.
\end{align}  
Now we combine the estimates for the $M_k$ to prove an $L^2$ estimate for $M$. For each $(i, k)$ where $1\le i\le N$ and $k\in\mathbb{Z}$, define a maximal operator $T^{i, k}$ by
\begin{align*}
T^{i, k}f(x)=\sup_{(2^{-k})^Ax\in R\in\mathcal{R}_i}\frac{1}{|(2^k)^AR|}\int_{(2^k)^AR}|f(y)|\,dy
\end{align*}
where $\mathcal{R}_i$ denotes the collection of all rectangles with direction $\pi iN^{-1}$ and dimensions $\lambda\times N\lambda$. Fix a Schwartz function $f\in\mathcal{S}(\mathbb{R}^2)$, and apply a standard covering lemma to obtain for each $(i, k)$ a sequence of rectangles $\{R^{i, k}_n\}\subset\mathcal{R}_i$ pairwise disjoint such that
\begin{align*}
E_{\alpha}^{i, k}=\{x: T^{i, k}f(x)>\alpha\}\subset \bigcup_n ((2^k)^A(R_n^{i, k})^{\ast}),
\end{align*}
\begin{align*}
\frac{1}{|(2^k)^AR^{i, k}_n|}\int_{(2^k)^AR^{i, k}_n}|f(y)|\,dy>\alpha.
\end{align*}
Then
\begin{align*}
E_{\alpha}=\{x: M_{\lambda, N}f(x)>4\alpha\}\subset\bigcup_{i, k}E_{\alpha}^{i, k}.
\end{align*}
Let 
\begin{align*}\mathcal{H}=\bigcup_{i, k, n}(2^k)^AR_{n}^{i, k}.
\end{align*}
Let $\mathcal{H}^{\prime}$ be a subcollection of $\mathcal{H}$ such that 
\begin{enumerate}
\item There are no $R, R^{\prime}\in\mathcal{H}^{\prime}$ such that $R^{\prime}\subset R^{\ast}$.
\item If $R\in\mathcal{H}\setminus\mathcal{H}^{\prime}$, then there is $R^{\prime}\in\mathcal{H}^{\prime}$ such that $R\subset (R^{\prime})^{\ast}$. 
\end{enumerate}
(To see that such a subcollection exists, we simply enumerate the rectangles in $\mathcal{H}$ as $R_1, R_2, \ldots$, and at step $i$ we add $R_i$ to $\mathcal{H}^{\prime}$ if $R_i$ is not contained in $R_j^{\ast}$ for any $j<i$ such that $R_j\in\mathcal{H}^{\prime}$, and in this case if $R_j\subset R_i^{\ast}$ for any $j<i$ such that $R_j\in\mathcal{H}^{\prime}$, we remove $R_j$ from $\mathcal{H}^{\prime}$.) Then 
\begin{align}\label{mf1}
E_{\alpha}\subset\bigcup_{R\in\mathcal{H}^{\prime}}R^{\ast\ast}.
\end{align} 
Let $\mathcal{H}_k=\mathcal{H}^{\prime}\cap\mathcal{C}_k$. Fix an integer $a>0$ such that $B(0, 2)\subset(2^a)^AB(0, 1)$, where $B(0, r)$ denotes the (Euclidean) ball of radius $r$ centered at the origin. Let $n_0=\max\{k:\,\mathcal{H}_k\ne\emptyset\}.$ For every $j\ge 0$, let 
\begin{align*}
\Delta_j=\bigcup_{\substack{n_0-(j+1)(\log N)^a\\< k\le n_0-j(\log N)^a}}\mathcal{H}_k.
\end{align*}
For each $j$ let $A_j=\bigcup_{R\in\Delta_j}R$. Then the family of sets $\{A_j\}$ is ``almost disjoint", i.e. $A_{j_1}\cap A_{j_2}=\emptyset$ if $|j_1-j_2|>2$. To see this, suppose that $R\in\Delta_{j_1}$ and $R^{\prime}\in \Delta_{j_2}$ with $j_1<j_2-2$ and $R\cap R^{\prime}\ne\emptyset$. Choose $k$ such that $R\in\mathcal{C}_k$. Then $(2^{-k})^AR\subset B(x, N\lambda)$ for some $x\in (2^{-k})^AR^{\prime}$. But $((2^{-k})^AR^{\prime})^{\ast}\supset B(x, N\lambda)$, and so $R\subset R^{\prime\ast}$, a contradiction.
\newline\indent
Now, by (\ref{mf1}) we have 
\begin{align}\label{mf2}
E_{\alpha}\subset\bigcup_jA_j^{\ast\ast}.
\end{align} 
Let $f_j=f\cdot\chi_{A_j}$.
Define a maximal function $S_j$ for $g\in\mathcal{S}(\mathbb{R}^2)$ by
\begin{align*}
S_jg(x)=\sup_{x\in R\in\bigcup_{\substack{n_0+2-(j+1)(\log N)^a\\<k\le n_0+2-j(\log N)^a}}\mathcal{C}_k}\frac{1}{|R|}\int_{R}g(y)\,dy.
\end{align*}
It follows from (\ref{kmax1}) that $S_j$ is bounded on $L^2(\mathbb{R}^2)$ with operator norm $\lesssim(\log N)^{a+1/2}$. Now if $x\in A_j^{\ast\ast}$, then there is $R\in\Delta_j$ such that $x\in R^{\ast\ast}$. Then,
\begin{align*}
S_jf_j(x)\ge\frac{1}{|R^{\ast\ast}|}\int_{R^{\ast\ast}}|f_j(y)|\,dy\ge\frac{1}{16}\frac{1}{|R|}\int_R|f_j(y)|\,dy\ge\frac{1}{16}\alpha.
\end{align*}
Thus $A_j^{\ast\ast}\subset\{x:\,S_jf_j(x)\ge\frac{1}{16}\alpha\}$, and so 
\begin{align*}
|A_j^{\ast\ast}|\lesssim(\log N)^{2a+1}\frac{\Norm{f_j}_2^2}{\alpha^2}.
\end{align*}
It follows that
\begin{align}\label{weak2}
|E_{\alpha}|\le\sum_j|A_j^{\ast\ast}|\lesssim(\log N)^{2a+1}\frac{1}{\alpha^2}\sum_j\Norm{f_j}_2^2\lesssim(\log N)^{2a+1}\frac{\Norm{f}_2^2}{\alpha^2}.
\end{align}
To obtain a strong type $L^2$ estimate for $M_{\lambda, N}$ from (\ref{weak2}), we will need to first prove a weak $(1, 1)$ estimate for $M_{\lambda, N}$ and interpolate. By comparison with the Hardy-Littlewood maximal function and rescaling, we have for every $k$,
\begin{align}\label{weak}
|\{x: M_{\lambda, N, k}(f)(x)>\alpha\}|\lesssim N\frac{\Norm{f}_1}{\alpha}.
\end{align}
We now repeat the above argument, using (\ref{weak}) in place of (\ref{kmax1}) and obtain the weak $(1, 1)$ estimate
\begin{align}\label{weak1}
|\{x: M_{\lambda, N}f(x)>4\alpha\}\lesssim N\frac{\Norm{f}_1}{\alpha}.
\end{align}
The result now follows by interpolation of (\ref{weak1}), (\ref{weak2}) and the trivial $L^{\infty}$ estimate for $M_{\lambda, N}$.
\end{proof}

\section{A decomposition of $\mathbb{R}^2$}

In this section, we will introduce a decomposition of $\mathbb{R}^2$ that plays a similar role as the decomposition of $\mathbb{R}^2$ provided in \cite{car2}. The decomposition from \cite{car2} can be viewed more or less as a decomposition of the annulus $|\xi-1|\le\delta$ into $\delta$-thickened caps that can be approximated by $\delta^{1/2}\times\delta$ rectangles, and dilated at different scales to cover the plane in an almost-disjoint fashion. Here we employ a different decomposition of the set $|\rho(\xi)-1|\le \delta$ into rectangles of width $\delta$ and length essentially between $\delta$ and $1$, so that on each rectangle, $\partial\Omega$ may be viewed as sufficiently ``flat" at scale $\delta$. This decomposition was introduced by \cite{sz} to prove $L^p$ bounds for Bochner-Riesz multipliers associated to convex domains. We then dilate these rectangles nonisotropically at different scales to cover the plane in an almost-disjoint fashion.
\subsection*{Decomposition of $\partial\Omega$} Before we describe the decomposition of $\mathbb{R}^2$, we first need to introduce a decomposition of $\partial\Omega$ from \cite{sz}. This decomposition allows us to write $\partial\Omega$ as a disjoint union of pieces on which $\partial\Omega$ is sufficiently ``flat". Here, the pieces in the decomposition will play the role that the $\delta^{1/2}$-caps play in the radial case. 
\newline
\indent
We inductively define a finite sequence of increasing numbers
\begin{align*}
\mathfrak{A}(\delta)=\{a_0, \ldots, a_Q\}
\end{align*}
as follows. Let $a_0=-1$, and suppose $a_0, \ldots, a_{l-1}$ are already defined. If
\begin{align}\label{case1}
(t-a_{l-1})(\gamma_L^{\prime}(t)-\gamma_R^{\prime}(a_{l-1}))\le\delta\text{ for all }t\in (a_{l-1}, 1])
\end{align}
and $a_{l-1}\le 1-2^{-M}\delta$, then let $a_l=1$. If (\ref{case1}) holds and $a_{l-1}>1-2^{-M}\delta$, then let $a_l=a_{l-1}+2^{-M}\delta$. If (\ref{case1}) does not hold, define
\begin{align*}
a_l=\inf\{t\in (a_{l-1}, 1]:\,(t-a_{l-1})(\gamma_L^{\prime}(t)-\gamma_R^{\prime}(a_{l-1}))>\delta\}.
\end{align*}
Now note that (\ref{case1}) must occur after a finite number of steps, since we have $|\gamma_L^{\prime}|, |\gamma_R^{\prime}|\le 2^{M-1}$, which implies that $|t-s||\gamma_L^{\prime}(t)-\gamma_R^{\prime}(s)|<\delta$ if $|t-s|<\delta 2^{-M}$. Therefore this process must end at some finite stage $l=Q$, and so it gives a sequence $a_0<a_1<\cdots<a_Q$ so that for $l=0, \ldots, Q-1$
\begin{align}\label{left} 
(a_{l+1}-a_l)(\gamma_L^{\prime}(a_{l+1})-\gamma_R^{\prime}(a_l))\le\delta,
\end{align}
and for $0\le j<Q-1$,
\begin{align}\label{right}
(t-a_l)(\gamma_L^{\prime}(t)-\gamma_R^{\prime}(a_l))>\delta\qquad\text{if }t>a_{l+1}.
\end{align}
For a given $\delta>0$, this gives a decomposition of 
\begin{align*}
\partial\Omega\cap\{x:\,-1\le x_1\le 1, x_2<0\}
\end{align*} 
into pieces
\begin{align*}
\bigsqcup_{l=0, 1, \ldots, Q-1}\{x\in\partial\Omega:\, x_1\in [a_l, a_{l+1}]\}.
\end{align*}
Now let $\{i_0, i_1, \ldots, a_{Q^{\prime}}\}$ be a refinement of $\{a_0, a_1, \ldots, a_Q\}$ corresponding to a partition of $[-1, 1]$ into intervals $\{I_j\}$ with $I_j=[i_j, b_{j+1}]$ such that each interval $[a_l, a_{l+1}]$ is a union of $\lesssim\log(\delta^{-1})$ of the intervals $I_j$, and so that $|I_{j}|/2\le |b_{j+1}|\le 2|I_j|$. We then have a decomposition
\begin{align*}
\partial\Omega\cap\{x:\,-1\le x_1\le 1, x_2<0\}=\bigsqcup_{j=0, 1, \ldots, Q^{\prime}}\{x\in\partial\Omega:\, x_1\in I_j\},
\end{align*}
where $Q^{\prime}\lesssim\log(\delta^{-1})Q$.

\subsection*{Decomposition of $\mathbb{R}^2$}
With the previous decomposition of $\partial\Omega$ in mind, we proceed to give a decomposition of $\mathbb{R}^2$. To begin, we define a \textit{nonisotropic sector} to be a region bounded by the origin and two orbits $\{t^A\xi:\,t>0\}$ and $\{t^A\xi^{\prime}:\,t>0\}$ for any $\xi, \xi^{\prime}\in\mathbb{R}^2\setminus\{0\}$. Observe there is an integer $N_M>0$ such that 
\begin{enumerate}
\item We can write $\mathbb{R}^2=\bigcup_{i=0}^{N_M}\mathcal{S}_i$, where each $\mathcal{S}_i$ is a nonisotropic sector and the $\mathcal{S}_i$ are essentially disjoint. 
\item For every $i$, there is a rotation $\mathcal{R}_i$ such that $\mathcal{R}_i(\partial\Omega\cap\mathcal{S}_i)\subset\{x:\,-1/2\le x_1\le 1/2\}$,  $\mathcal{R}_i(\partial\Omega\cap\mathcal{S}_j)\cap\{x:\,-1/2< x_1< 1/2\}=\emptyset$ for $i\ne j$, and $\mathcal{R}_0$ is the identity map. 
\end{enumerate}
For each $i$, define $\tilde{\mathcal{S}}_i$ to be the nonisotropic sector bounded by the orbits $\{t^A\xi_i:\,t>0\}$ and $\{t^A\xi_i^{\prime}:\,t>0\}$ where $\xi=(\xi_1, \xi_2)$ is the unique point in $\mathcal{R}_i\partial\Omega$ with $\xi_1=-1$ and $\xi_2>0$, and $\xi^{\prime}=(\xi_1^{\prime}, \xi_2^{\prime})$ is the unique point in $\mathcal{R}_i\partial\Omega$ with $\xi_1=1$ and $\xi_2>0$. Clearly, $\mathcal{S}_i\subset\tilde{\mathcal{S}}_i$. Let $\{(\alpha, \gamma_i(\alpha)):\,\alpha\in [-1, 1]\}$ be a parametrization of $\mathcal{R}_i(\partial\Omega\cap\tilde{\mathcal{S}}_i)$. For $0\le i\le N_M$, let $R_i$ denote the region of $\mathbb{R}^2$ bounded by the level sets $\{x:\,\rho(x)=1/2\}$ and $\{x:\,\rho(x)=2\}$ and the nonisotropic sector $\tilde{\mathcal{S}}_i$. Similarly, let $R_i^{\prime}$ denote the region of $\mathbb{R}^2$ bounded by the level sets $\{x:\,\rho(x)=1/4\}$ and $\{x:\,\rho(x)=4\}$ and the nonisotropic sector $\tilde{\mathcal{S}}_i$. Fix $\delta>0$. Let $R_{i, \delta}$ denote the region bounded by the level sets $\{x:\,\rho(x)=1-2\delta\}$ and $\{x:\,\rho(x)=1+2\delta\}$ and $\tilde{\mathcal{S}}_i$. Note that $\bigcup_{i=0}^{N_M}R_{i, \delta}$ contains the support of $\mathcal{F}[\psi_1]$, where $\psi_1$ is as in Proposition \ref{mainprop}.
\newline
\indent
Recall the previous decomposition of $[-1, 1]$ into intervals $\{I_j\}$. Let $B_{i, j, 0, 0}$ denote the region bounded by $R_{i, \delta}$ and the orbits $\{t^A\mathcal{R}_i^{-1}(i_{j}, 
\\\gamma_i(i_j)
):\,t>0\}$ and $\{t^A\mathcal{R}_i^{-1}(b_{j+1}, \gamma_i(b_{j+1})):\,t>0\}$, so that $R_{i, \delta}=\bigcup_j B_{i, j, 0, 0}$. For each integer $m$, let $B_{i, j, m, 0}=(\frac{1+2\delta}{1-2\delta})^{mA}B_{i, j, 0, 0}$. Now let $N_{\delta}, N_{\delta}^{\prime}$ be integers such that 
\begin{align*}
R_i\subset\bigcup_{j,   N_{\delta}\le m\le N_{\delta}^{\prime}}B_{i, j, m, 0}\subset R_i^{\prime}.
\end{align*}
We are now ready to state our decomposition of $\mathbb{R}^2$. For each integer $n$, let $B_{i, j, m, n}=(2^{2n})^A(B_{i, j, m, 0}).$ Then 
\begin{align*}
\mathcal{S}_i\subset\bigcup_{j, N_{\delta}\le m\le N_{\delta}^{\prime}, n\in\mathbb{Z}} B_{i, j, m, n},
\end{align*}
\begin{align*}
\mathbb{R}^2=\bigcup_{0\le i\le N_M, j, N_{\delta}\le m\le N_{\delta}^{\prime}, n\in\mathbb{Z}} B_{i, j, m, n},
\end{align*}
and there is an integer $N_M^{\prime}$ depending only on $M$ such that every point of $\mathbb{R}^2$ lies in at most $N_M^{\prime}$ many elements of the collection $\{B_{i, j, m, n}\}$.

\subsection*{Some important properties of the decomposition.}
We now prove some essential geometric facts regarding our decomposition; these may be viewed as analogs of Proposition $3$ parts $(i)-(iii)$ from \cite{car2}. The following proposition is a key fact regarding the almost disjointness of algebraic sums of the pieces in our decomposition.
\begin{proposition}\label{algdisj}
For a constant $C(M, \text{Re}(\lambda_1), \text{Re}(\lambda_2))>0$ depending only on $M$ and the eigenvalues of $A$, let
\begin{align*}
\mathcal{T}_0=\{\xi:\,1/4\le\rho(\xi)\le 4, |\xi_1|\le C(M, \text{Re}(\lambda_1), \text{Re}(\lambda_2))\},
\end{align*}
\begin{align*}
\mathcal{T}_1=\bigcup_{k\in\mathbb{Z}}(2^{4k})^A(\mathcal{T}_0).
\end{align*}
For $0<t<\infty$, let 
\begin{align*}
\mathcal{A}_t=\{B\in\{B_{i, j, m, n}\}:\exists\xi\in B\text{ with }\rho(\xi)=1/t
\\
\text{ and }\xi\in\mathcal{T}_1\}.
\end{align*}
Fix positive real numbers $u$ and $t$ satisfying $1/2<u/t<2$ with $u\in\bigcup_{k\in\mathbb{Z}}[2^{4k-1}, 2^{4k+1}]$, and let $\mathcal{B}_{u, t}$ denote the collection of all sets of the form $\{A+B\}_{A\in\mathcal{A}_t, B\in\mathcal{A}_u}$. Then if $C(M, \text{Re}(\lambda_1), \text{Re}(\lambda_2))$ is chosen sufficiently small, there exists a constant $C^{\prime}(M, \text{Re}(\lambda_1), \text{Re}(\lambda_2))>0$ (depending only on $M$ and the eigenvalues of $A$ and independent of $\delta$ and the choice of $u$ and $t$) such that every point of $\mathbb{R}^2$ is contained in at most 
\begin{align*}C^{\prime}(M, \text{Re}(\lambda_1), \text{Re}(\lambda_2))(\log(\delta^{-1}))^2
\end{align*} 
elements of $\mathcal{B}_{u, t}$.
\end{proposition}

\begin{proof}
Without loss of generality, assume that $u=1$. For any $v\in A\in\mathcal{A}_t$ and $w\in B\in\mathcal{A}_u$, let $\sigma^+(v, w)$ denote the minimum nonnegative difference in slope between supporting lines to the convex curve 
\begin{align*}
\Sigma_v:=\{\xi:\,\rho(\xi)=\rho(v),\,\,\xi\in\mathcal{T}_1\}
\end{align*}
at $v$ and supporting lines to the convex curve
\begin{align*}
\Sigma_w:=\{\xi:\,\rho(\xi)=\rho(w),\,\,\xi\in\mathcal{T}_1\}
\end{align*}
at $w$, and $\sigma^+(v, w):=+\infty$ if no nonnegative difference exists. Let $\sigma^-(v, w)$ denote the maximum nonpositive difference in slope between supporting lines to $\Sigma_v$ at $v$ and supporting lines to $\Sigma_w$ at $w$, and $\sigma^-(v, w):=-\infty$ if no nonpositive difference exists. Note that for every $(v, w)$ at least one of $\sigma^+(v, w)$ and $\sigma^-(v, w)$ is finite, and if $C(M, \text{Re}(\lambda_1), \text{Re}(\lambda_2))$ is sufficiently small, then the slope of any supporting line is between $-2^{2M}$ and $2^{2M}$. Given $x\in\mathcal{B}_{u, t}$, we have one of three cases:
\begin{enumerate}
\item There is $v\in A\in\mathcal{A}_t$ and $w\in B\in\mathcal{A}_u$ with $v+w=x$ and $\sigma^+(v, w)$ finite, but $\sigma^-(v, w)$ is infinite for every pair $(v^{\prime}, w^{\prime})$ with $v\in A^{\prime}\in\mathcal{A}_t$, $w\in B^{\prime}\in\mathcal{A}_u$, and $v^{\prime}+w^{\prime}=x$,
\item There is $v\in A\in\mathcal{A}_t$ and $w\in B\in\mathcal{A}_u$ with $v+w=x$ and $\sigma^-(v, w)$ finite, but $\sigma^+(v, w)$ is infinite for every pair $(v^{\prime}, w^{\prime})$ with $v\in A^{\prime}\in\mathcal{A}_t$, $w\in B^{\prime}\in\mathcal{A}_u$, and $v^{\prime}+w^{\prime}=x$,
\item There is $v\in A\in\mathcal{A}_t$ and $w\in B\in\mathcal{A}_u$ with $v+w=x$ and $\sigma^+(v, w)$ finite, and there is $v^{\prime}\in A^{\prime}\in\mathcal{A}_t$ and $w^{\prime}\in B^{\prime}\in\mathcal{A}_u$ with $v^{\prime}+w^{\prime}=x$ and $\sigma^-(v, w)$ finite.
\end{enumerate}

Let us assume we have case $1$. Given $x\in\mathbb{R}^2$, choose $v=(v_1, v_2)\in A\in\mathcal{A}_t$ and $w=(w_1, w_2)\in B\in\mathcal{A}_u$ with $v+w=x$ minimizing $\sigma^+(v, w)$. Now suppose there is $\tilde{v}=(\tilde{v}_1, \tilde{v}_2)\in\tilde{A}\in\mathcal{A}_t$ and $\tilde{w}=(\tilde{w}_1, \tilde{w}_2)\in\tilde{B}\in\mathcal{B}_t$ such that $\tilde{v}+\tilde{w}=x$. Since $\Sigma_v$ and $\Sigma_w$ are convex, we have 
\begin{align*}
\tilde{v}_1\le v_1+C^{\prime\prime}(M, \text{Re}(\lambda_1), \text{Re}(\lambda_2))\delta,
\end{align*}
\begin{align*}
\tilde{w}_1\ge w_1-C^{\prime\prime}(M, \text{Re}(\lambda_1), \text{Re}(\lambda_2))\delta,
\end{align*} 
where $C^{\prime\prime}(M, \text{Re}(\lambda_1), \text{Re}(\lambda_2))>0$ is a constant that depends only on $M$ and the eigenvalues of $A$. Thus 
\begin{align}\label{diffest}
v_1-\tilde{v}_1=\tilde{w}_1-w_1\ge -C^{\prime\prime}(M, \text{Re}(\lambda_1), \text{Re}(\lambda_2))\delta.
\end{align} 
Choose indices $i_0, j_0, m_0, n_0$ and $i_0^{\prime}, j_0^{\prime}, m_0^{\prime}, n_0^{\prime}$ such that $B_{i_0, j_0, m_0, n_0}\ni v$ and $B_{i_0^{\prime}, j_0^{\prime}, m_0^{\prime}, n_0^{\prime}}\ni w$. (There are $\lesssim_{M, \text{Re}(\lambda_1), \text{Re}(\lambda_2)}$ possible choices of indices.) Also choose indices $i_1, j_1, m_1$ and $i_1^{\prime}, j_1^{\prime}, m_1^{\prime}$ such that $B_{i_1, j_1, m_1, n_1}\ni \tilde{v}$ and $B_{i_1^{\prime}, j_1^{\prime}, m_1^{\prime}, n_1^{\prime}}\ni \tilde{w}$. Note that we must necessarily have $m_1=m_0$ and $m_1^{\prime}=m_0^{\prime}$, and also that $-10\le n_1, n_1^{\prime}\le 10$. We next observe that for some sufficiently large constant $C^{\prime\prime\prime}(M, \text{Re}(\lambda_1), \text{Re}(\lambda_2))$ we must have 
\begin{multline*}
j_0-C^{\prime\prime\prime}(M, \text{Re}(\lambda_1), \text{Re}(\lambda_2))\log(\delta^{-1})^2\le j_1
\\
\le j_0+C^{\prime\prime\prime}(M, \text{Re}(\lambda_1), \text{Re}(\lambda_2))\log(\delta^{-1})^2,
\end{multline*}
\begin{multline*}
j_0^{\prime}-C^{\prime\prime\prime}(M, \text{Re}(\lambda_1), \text{Re}(\lambda_2))\log(\delta^{-1})^2\le j_1^{\prime}
\\
\le j_0^{\prime}+C^{\prime\prime\prime}(M, \text{Re}(\lambda_1), \text{Re}(\lambda_2))\log(\delta^{-1})^2,
\end{multline*} 
since otherwise (\ref{right}) and (\ref{diffest}) would imply that $\tilde{v}_2+\tilde{w}_2<v_2+w_2-C^{\prime\prime\prime\prime}(M, \text{Re}(\lambda_1), \text{Re}(\lambda_2)) \delta$. This completes the proof for case $1$, since we have shown that for some constant $C(M, \text{Re}(\lambda_1), \text{Re}(\lambda_2))$ sufficiently large there are fewer than $C(M, \text{Re}(\lambda_1), \text{Re}(\lambda_2))\log(\delta^{-1})^2$ possible choices of indices $i_1, j_1, m_1, n_1$ and $i_1^{\prime}, j_1^{\prime}, m_1^{\prime}, n_1^{\prime}$ such that $B_{i_1, j_1, m_1, n_1}\ni\tilde{v}$ and $B_{i_1^{\prime}, j_1^{\prime}, m_1^{\prime}, n_1^{\prime}}\ni\tilde{w}$. The proof for case $2$ is similar.
\newline
\indent
Now let us assume we have case $3$. Suppose there is $\tilde{v}=(\tilde{v}_1, \tilde{v}_2)\in\tilde{A}\in\mathcal{A}_t$ and $\tilde{w}=(\tilde{w}_1, \tilde{w}_2)\in\tilde{B}\in\mathcal{B}_t$ such that $\tilde{v}+\tilde{w}=x$. Then if $\sigma^+(\tilde{v}, \tilde{w})$ is finite, then there is a constant $C^{\prime}(M, \text{Re}(\lambda_1), \text{Re}(\lambda_2))>0$ such that
\begin{align*}
\tilde{v}_1\le v_1+C^{\prime}(M, \text{Re}(\lambda_1), \text{Re}(\lambda_2))\delta,\qquad\tilde{w}_1\ge w_1-C^{\prime}(b_jM, \text{Re}(\lambda_1), \text{Re}(\lambda_2))\delta, 
\end{align*}
and if $\sigma^-(\tilde{v}, \tilde{w})$ is finite, then
\begin{align*}
\tilde{v}_1\le v_1+C^{\prime}(M, \text{Re}(\lambda_1), \text{Re}(\lambda_2))\delta,\qquad\tilde{w}_1\ge w_1-C^{\prime}(M, \text{Re}(\lambda_1), \text{Re}(\lambda_2))\delta.
\end{align*}
In either case, the previous argument shows there is a constant $C=C(M, 
\\
\text{Re}(\lambda_1), \text{Re}(\lambda_2))>0$ such that there are fewer than $C\log(\delta^{-1})^2$ possible choices of indices $i_1, j_1, m_1, n_1$ and $i_1^{\prime}, j_1^{\prime}, m_1^{\prime}, n_1^{\prime}$ such that $B_{i_1, j_1, m_1, n_1}
\\\ni\tilde{v}$ and $B_{i_1^{\prime}, j_1^{\prime}, m_1^{\prime}, n_1^{\prime}}\ni\tilde{w}$.

\end{proof}

\begin{proposition}\label{imp2}
Let $N(M, \text{Re}(\lambda_1), \text{Re}(\lambda_2))$ be a positive integer and let $\delta>0$, and fix positive real numbers $u$ and $t$ satisfying $\delta^{N(M, \text{Re}(\lambda_1), \text{Re}(\lambda_2))} t>u$. Then if $N(M, \text{Re}(\lambda_1), \text{Re}(\lambda_2))$ is sufficiently large, there exists a constant $C(M, \text{Re}(\lambda_1), \text{Re}(\lambda_2))>0$ (independent of $\delta$ and the choice of $u$ and $t$) such that no point of $\mathbb{R}^2$ is contained in more than $C(M, \text{Re}(\lambda_1), \text{Re}(\lambda_2))$ of the sets $\{A+B_{\rho}(0, 2/t)\}_{A\in\mathcal{A}_u}$, where $B_{\rho}(0, r)=\{x\in\mathbb{R}^2: \rho(x)\le r\}$.
\end{proposition}

\begin{proof}
Without loss of generality, suppose that $u=1$. Fix $A\in\mathcal{A}_u$, and let $x\in A$ and let $y\in B_{\rho}(0, 2/t)$. Choose $N(M, \text{Re}(\lambda_1), \text{Re}(\lambda_2))$ large enough to make $B_{\rho}(0, 2/t)\subset B(0, \delta^2)$, where $B(0, \delta^2)$ denotes the (Euclidean) ball of radius $\delta^2$ centered at the origin. Assume $\delta<C^{\prime}(M, \text{Re}(\lambda_1), \text{Re}(\lambda_2), \Theta(\Omega, A))$, where $C^{\prime}(M, \text{Re}(\lambda_1), \text{Re}(\lambda_2), \Theta(\Omega, A))>0$ is chosen sufficiently small so that the minimum angle between the tangent line to $\xi\in\partial\Omega$ and any tangent line to the curve $\{t^A\xi:\,1-10\delta\le t\le1+10\delta\}$ is at least $\delta^{1/2}$. Now for any $\xi\in\partial\Omega$, $1-10\delta\le t\le 1+10\delta$, we have 
\begin{align*}
\bigg|\frac{d}{dt}(t^A\xi)\bigg|=|t^{-1}At^A\xi|\gtrsim_{M, \text{Re}(\lambda_1), \text{Re}(\lambda_2)}1,
\end{align*}
and it follows that if $C^{\prime}(M, \text{Re}(\lambda_1), \text{Re}(\lambda_2), \Theta(\Omega, A))$ is sufficiently small, the (Euclidean) distance between $t^A\xi$ and the tangent line to $\partial\Omega$ at $\xi$ is at least $10\delta^2$. Since $\Omega$ is convex, we conclude that the distance between $\partial\Omega$ and $(1+\delta)^A\partial\Omega$ is at least $10\delta^2$. Similarly, the distance between $\partial\Omega$ and $(1-\delta)^A\partial\Omega$ is at least $10\delta^2$. It follows there is an absolute constant $C$ such that for any given $\xi\in\mathbb{R}^2$, there are fewer than $C$ possible values of $m$ (and clearly also fewer than $C$ possible values of $n$) such that $B_{i, j, m, n}+B(0, \delta^2)\ni\xi$ for some $B_{i, j, m, n}\in\mathcal{A}_1$. It remains to obtain an upper bound for the number of possible values of $j$. But it is clear that dist$(B_{i, j, m, n}, B_{i, j^{\prime}, m, n})\ge\delta/10$ for $|j-j^{\prime}|>2$, and this finishes the proof.

%We observe that for any $c>0$,
%$$\left<\delta_{1/(\rho(x)+c\rho(y))}
%(x+y)\right>\le\left<\delta_{\rho(x)/(\rho(x)+c\rho(y))}w(x)\right>+\left<\delta_{\rho(y)/(\rho(x)+c\rho(y))}w(y)\right>\le %1+c^{\prime}c^{-\alpha}$$
%for some positive $\alpha$. So $\rho(x+y)\le (1+c^{\prime}c^{-\alpha})(\rho(x)+c\rho(y))$. Now take %$c=\delta^{-M}$, where $M$ can be taken large enough to make $(1+c^{\prime}c^{-\alpha})$ smaller than %$(1+\delta^K)$ for any arbitrary $K$. Next, choose $N$ large enough to make $|B(0, 2/t)|%\lesssim\delta^2$ and $(\rho(x)+c\rho(y))\le (1+\delta^K)$. Then $\rho(x+y)\le(1+\delta^K)\rho(x)$. By a %similar argument, we also have $\rho(x+y)\ge (1-\delta^K)\rho(x)$. Taking $K$ sufficiently large, the %result follows.

\end{proof}

\begin{proposition}\label{imp3}
There exists an absolute constant $C>0$ such that for each fixed quadruple $(i, j, m, n)$, the logarithmic measure of $\{t: B_{i, j, m, n}\cap\text{supp\,}\mathcal{F}[\psi_t]\ne\emptyset\}$ is less than or equal to $C\delta$.
\end{proposition}

\begin{proof}
Immediate.
\end{proof}

\section{Kernel estimates and another $L^2$ maximal function estimate}\label{kerest}
We note that in both \cite{cor1} and \cite{car2}, it was important that regarding the decomposition of the multiplier $\phi(\delta^{-1}(1-|\xi|)$ where $\phi$ was a smooth bump function into pieces supported on $\delta^{1/2}\times\delta$ rectangles, each piece of the multiplier had $L^1$ norm essentially $1$. This was also true of the decomposition of $|\rho(\xi)-1|\le\delta$ introduced in \cite{sz}. In this section we prove that after the introduction of nonisotropic dilations, the same holds true. 
\newline
\indent
The argument presented in \cite{car2} also used $L^2$ bounds for maximal functions given by the supremum of convolutions by smooth bumps supported on finitely many essentially disjoint pieces of the decomposition of $\mathbb{R}^2$ given in \cite{car2}. Since these smooth bumps could be dominated by Schwartz functions adapted to rectangles, such a maximal function could be dominated by a Nikodym maximal function. Here, as well as in \cite{sz}, we do not have domination of the functions in our partition of unity by Schwartz functions adapted to rectangles, and the proof of $L^1$ kernel estimates is more delicate. As in \cite{sz}, this also implies that the associated maximal function that we use is not simply a nonisotropic Nikodym maximal function. However, we will show that the $L^2$ bounds for the nonisotropic Nikodym maximal function proved earlier imply $L^2$ bounds for the maximal function that we are interested in, with a similar constant.

\subsection*{A partition of unity associated to the decomposition of $\mathbb{R}^2$}
First, we need to define a partition of unity of $\mathbb{R}^2$, and as mentioned above one goal of this section is to show that each function in our partition of unity has bounded $L^1$ norm. Recall the decomposition
\begin{align*}
\mathbb{R}^2=\bigcup_{i, j, m, n}B_{i, j, m, n}.
\end{align*}
We now introduce a partition of unity $\{\sigma_{i, j, m, n}\}$ such that 
\begin{enumerate}
\item $\sigma_{i, j, m, n}\in C^{\infty}(\mathbb{R}^2)$ for every $(i, j, m, n)$,
\item $\sum_{i, j, m, n}\sigma_{i, j, m, n}(x)=1$ for every $x\in\mathbb{R}^2$,
\item There is a constant $C_M$ such that for every $(i_0, j_0, m_0, n_0)$, $\sigma_{i_0, j_0, m_0, n_0}$ is supported in $\bigcup_{|j|, |m|\le C_M}B_{i_0, j_0+j, m_0+m, n_0}$.
\end{enumerate}
Let $\phi\in C^{\infty}([-1, 1])$ be nonnegative and identically $1$ on $[-1/2, 1/2]$, and for $n\in\mathbb{Z}$ set $\phi_n(\cdot)=\phi(2^{-n-1}\cdot)-\phi(2^{-n}\cdot)$. For each $m$, let $\psi_m\in C^{\infty}(1-(2m+10)\delta, 1+(2m+10)\delta)$ such that $\sum_m\psi_m$ is identically $1$ on the support of $\phi_0$, and for every $k$, $D^k\psi_m\lesssim_k \delta^{-k}$.
\newline
\indent
For each $i$, let $S_i$ be the \textit{isotropic} sector bounded by $|\xi|=2, |\xi|=2^{M+2}$, and the rays through the origin and the points $\xi$ and $\xi^{\prime}$, where $\xi=(\xi_1, \xi_2)$ is the unique point in $\mathcal{R}_i\partial\Omega$ with $\xi_1=-1/4$ and $\xi_2>0$, and $\xi^{\prime}=(\xi_1^{\prime}, \xi_2^{\prime})$ is the unique point in $\mathcal{R}_i\partial\Omega$ with $\xi_1=1/4$ and $\xi_2>0$. Let $\tilde{S}_i$ be the isotropic sector bounded by  $|\xi|=1, |\xi|=2^{M+3}$, and the rays through the origin and the points $\xi$ and $\xi^{\prime}$, where $\xi=(\xi_1, \xi_2)$ is the unique point in $\mathcal{R}_i\partial\Omega$ with $\xi_1=-3/4$ and $\xi_2>0$, and $\xi^{\prime}=(\xi_1^{\prime}, \xi_2^{\prime})$ is the unique point in $\mathcal{R}_i\partial\Omega$ with $\xi_1=3/4$ and $\xi_2>0$.  For each $i$, let $\Psi_i$ be a smooth function supported in $\tilde{S}_i$ and identically $1$ on $S_i$, such that $D^k\Psi_i\lesssim_{M, k} 1$ for all $k$ and $\sum_i\Psi_i$ is identically $1$ on the region bounded by $|\xi|=2$ and $|\xi|=2^{M+2}$.  
\newline
\indent
Fix $i$, and for each $j$, let $\ell_{j-1}$, $\ell_j$, and $\ell_{j+1}$ be the lines through $(b_{j-1}, \gamma_i
\\
(b_{j-1}))$, $(i_j, \gamma_i(i_j))$, and $(b_{j+1}, \gamma_i(b_{j+1}))$, respectively, with slopes orthogonal to the tangent vectors $(1, \gamma_i^{\prime}(b_{j-1})$, $(1, \gamma_i^{\prime}(i_j))$, and $(1, \gamma_i^{\prime}(b_{j+1}))$, respectively. Let $e_j$ be a unit vector orthogonal to $\ell_j$. Let $\alpha$ be a $C^{\infty}(\mathbb{R})$ function such that $0\le\alpha\le 1$, $\alpha(x)=1$ for $x\in [-1, 1]$ and $\alpha(x)=0$ for $x\notin [-\frac{101}{100}, \frac{101}{100}]$, and set $\alpha_j(\xi)=\alpha(|I_j|^{-1}(\xi-(i_j, \gamma_i(i_j)))\cdot e_j)$.
We are now ready to define the functions $\sigma_{i, j, m, n}$. Let
\begin{multline}\label{sigmadef}
\sigma_{i, j, m, 0}(\xi)=\phi_0(\rho(\xi))\Psi_i((\frac{1-2\delta}{1+2\delta})^{mA}\xi)\psi_m(\rho(\xi))
\\
\times\alpha_j(\mathcal{R}_i(\frac{1-2\delta}{1+2\delta})^{mA}\xi)(1-\alpha_{j+1}(\mathcal{R}_i(\frac{1-2\delta}{1+2\delta})^{mA}\xi)),
\end{multline}
and 
\begin{align}\label{sigmadef2}
\sigma_{i, j, m, n}(\xi)=\sigma_{i, j, m, 0}((2^{-n})^A\xi).
\end{align}
For every $i$ and every $m$, we have
\begin{align*}
\sum_{j}\alpha_j(\mathcal{R}_i(\frac{1-2\delta}{1+2\delta})^{mA}\xi)(1-\alpha_{j+1}(\mathcal{R}_i(\frac{1-2\delta}{1+2\delta})^{mA}\xi))
\end{align*}
is identically $1$ on the support of 
\begin{align*}
\phi_0(\rho(\xi))\Psi_i((\frac{1-2\delta}{1+2\delta})^{mA}\xi)\psi_m(\rho(\xi)),
\end{align*}
and since 
\begin{align*}
\sum_i\sum_m\phi_0(\rho(\xi))\Psi_i((\frac{1-2\delta}{1+2\delta})^{mA}\xi)\psi_m(\rho(\xi))
\\=\sum_m\phi_0(\rho(\xi))\psi_m(\rho(\xi))=\phi_0(\rho(\xi)),
\end{align*}
it follows that for every $\xi\in\mathbb{R}^2$,
\begin{align*}
\sum_{i, j, m, n}\sigma_{i, j, m, n}(\xi)=1.
\end{align*}
\iffalse{

\begin{align*}
\sigma_{i, j, m, 0}(\xi)=\phi_0(\rho(\xi))\Psi_i((\frac{1-2\delta}{1+2\delta})^{mA}\xi)\psi_m(\rho(\xi))\beta_j((\mathcal{R}_i(1/\rho(\xi))^A\xi)_1)
\end{align*}
and 
\begin{align*}
\sigma_{i, j, m, n}(\xi)=\sigma_{i, j, m, 0}((2^{-n})^A\xi).
\end{align*}
If $\delta>C(M, \text{Re}(\lambda_1), \text{Re}(\lambda_2))$ for a sufficiently large constant $C(M, \text{Re}(\lambda_1), \text{Re}(\lambda_2))$, then for every $i$,
\begin{align*}
\sum_{j}\beta_j((\mathcal{R}_i(1/\rho(\xi))^A\xi)_1)
\end{align*}
is identically $1$ on the support of 
\begin{align*}
\sum_m\phi_0(\rho(\xi))\Psi_i((\frac{1-2\delta}{1+2\delta})^{mA}\xi)\psi_m(\rho(\xi)),
\end{align*}
and since
\begin{align*}
\sum_i\sum_m\phi_0(\rho(\xi))\Psi_i((\frac{1-2\delta}{1+2\delta})^{mA}\xi)\psi_m(\rho(\xi))=\sum_m\phi_0(\rho(\xi))\psi_m(\rho(\xi))=\phi_0(\rho(\xi)),
\end{align*}
it follows that for every $\xi\in\mathbb{R}^2$,
\begin{align*}
\sum_{i, j, m, n}\sigma_{i, j, m, n}(\xi)=1.
\end{align*}

For each $j$, let $\beta_j\in C^{\infty}(\mathbb{R})$ such that $D^k\beta_j\lesssim_k|I_j|^{-k}$ for all $k$ and $\sum_j\beta_j$ is identically $1$ on $[-4/5, 4/5]$.
Define 
\begin{align*}
\tilde{\sigma}_{i, j, m, 0}(\xi)=\phi_0(\rho(\xi))\Psi_i((\frac{1-2\delta}{1+2\delta})^{mA}\xi)\psi_m(\rho(\xi))\beta_j((\mathcal{R}_i(1/\rho(\xi))^A\xi)_1)
\end{align*} 
\begin{align*}
\tilde{\sigma}_{i, j, m, n}(\xi)=\tilde{\sigma}_{i, j, m, 0}((2^{-n})^A\xi).
\end{align*}
}\fi

\subsection*{Introduction of a maximal function associated with the partition of unity}
Let
\begin{align*}
K_{i, j, m, n}(x)=\mathcal{F}[\sigma_{i, j, m, n}(\cdot)](x).
\end{align*}
We define a maximal function $\overline{M}$ on $f\in\mathcal{S}(\mathbb{R}^2)$
by
\begin{align*}
\overline{M}f(x)=\sup_{i, j, m, n}\sup_{2^{n-10}\le t\le 2^{n+10}}|\psi_t\ast K_{i, j, m, n}\ast f(x)|.
\end{align*}
We will prove the following $L^2$ bounds for $\overline{M}$.
\begin{proposition}\label{bigmaxprop}
Let $\epsilon>0$. There is a constant $C=C(M, \text{Re}(\lambda_1), \text{Re}(\lambda_2), 
\\
\Theta(\Omega, A))$ such that if $0<\delta<C$, then for $f\in\mathcal{S}(\mathbb{R}^2)$, 
\begin{align*}
\Norm{\overline{M}f}_{L^2(\mathbb{R}^2)}\lesssim_{\epsilon, M, \text{Re}(\lambda_1), \text{Re}(\lambda_2), \Theta(\Omega, A)}\delta^{-\epsilon}\Norm{f}_{L^2(\mathbb{R}^2)}.
\end{align*}
\end{proposition}

\begin{proof}
The proof will follow \cite{sz}. First note that without loss of generality we may drop the ``sup" in the $i$ index in the definition of $\overline{M}$ and assume $i=0$, and so in what follows we drop all $i$-indices. Set $l=\lceil{\log(\delta^{-1})}\rceil$. We decompose $\overline{M}=M_1+M_2$, where
\begin{align*}
M_1f(x)=\sup_{j, m, n}\sup_{2^{n-10}\le t\le 2^{n+10}}|\psi_t\ast (K_{j, m, n}\cdot\chi_{|t^A\cdot|\ge 2^{10M\cdot l}})\ast f(x)|,
\end{align*}
\begin{align*}
M_2f(x)=\sup_{j, m, n}\sup_{2^{n-10}\le t\le 2^{n+10}}|\psi_t\ast (K_{j, m, n}\cdot\chi_{|t^A\cdot|< 2^{10M\cdot l}})\ast f(x)|.
\end{align*}
We will first prove Proposition \ref{bigmaxprop} with $\overline{M}$ replaced by $M_1$. Let $\sigma_j(\xi)=\mathcal{F}^{-1}[K_{j, 0, 0}(\cdot)](\xi)$. Note that 
\begin{align}\label{mult1}
\sigma_j(\xi)=\phi_0(\rho(\xi))\psi_0(\rho(\xi))m_j(\xi),
\end{align}
where
\begin{multline}\label{mult2}
m_j(\xi)=\Psi_0(\xi)\phi_0(2^{-2M}\xi)\alpha_j(\xi)(1-\alpha_j(\xi-2^{M+10}|I_j|(1, \gamma_0^{\prime}(i_j))))
\\
(1-\alpha_{j+1}(\xi))(\alpha_{j+1}(\xi+2^{M+10}|I_j|(1, \gamma_0^{\prime}(b_{j+1})))).
\end{multline}
\indent Now let $\beta\in C^{\infty}$ be supported in $[-1, 1]$, and let $h_l(s)=\beta(2^{l}(1-s))$. Note that (\ref{mult1}) says that $\sigma_j$ is of the form $h_l(\rho(\cdot))m_j(\cdot)$. We claim that to prove Proposition \ref{bigmaxprop} with $M_1$ in place of $\overline{M}$, it in fact suffices to prove Proposition \ref{bigmaxprop} with $\overline{M}f$ replaced by
\begin{align*}
\sup_{t\in(0, \infty)}|(\chi_{|t^{A}\cdot|\ge 2^{5M\cdot l}}\cdot\mathcal{F}^{-1}[h_l(t\rho(\cdot))])\ast f(x)|.
\end{align*}
This will follow immediately from the observation that $2^{M+10}|I_j|^{-1}<<2^{10M\cdot l}$ and that for any annulus $\mathcal{A}_k$,
\begin{align}\label{intann}
\int_{\mathcal{A}_k}\mathcal{F}^{-1}[h_l(\rho(\cdot))](x)\,dx\lesssim 1,
\end{align} 
which will be proven later.
\newline
\indent
We now prove pointwise estimates for $\mathcal{F}^{-1}[h_l(\rho(\cdot))](x)$, which we write as an integral over $\partial\Omega$ as follows:
\begin{align*}
(2\pi)^2\mathcal{F}^{-1}[h_l(\rho(\cdot))](x)=\int_{\Omega}h_l(\rho(\xi))e^{i\left<x, \xi\right>}\,d\xi=-\int_{\Omega}e^{i\left<x, \xi\right>}\int_{\rho(\xi)}^{\infty}h_l^{\prime}(s)\,ds\,d\xi
\end{align*}
\begin{align*}
=-\int_0^{\infty}h_l^{\prime}(s)\int_{\rho(\xi)\le s}e^{i\left<x, \xi\right>}\,d\xi\,ds=-\int_0^{\infty}h_l^{\prime}(s)\int_{\rho(\xi)\le 1}e^{i\left<x, s^A\xi\right>}|\det s^A|\,d\xi\,dx
\end{align*}
\begin{align*}
=\int_0^{\infty}i|s^{A^{\ast}}x|^{-2}h_l^{\prime}(s)\int_{\partial\Omega}e^{i\left<s^{A^{\ast}}x, \xi\right>}\left<s^{A^{\ast}}x, n(\xi)\right>\,d\sigma(\xi)|\det s^A|\,ds.
\end{align*}
In the above computation, we used the divergence theorem applied to the vector field $\xi\mapsto (i|s^{A^{\ast}}x|^2)^{-1}s^{A^{\ast}}xe^{i\left<s^{A^{\ast}}x, \xi\right>}.$ For each $i$, let $\zeta_i\in C^{\infty}(\mathbb{R})$ be supported in $[-4/5, 4/5]$ and identically $1$ on $[-1/3, 1/3]$ such that $\sum_i\zeta_i((\mathcal{R}_i
\\(1/\rho(\xi))^A\xi)_1)\equiv 1$. It suffices to estimate
\begin{multline}\label{int1}
\int_0^{\infty}i|s^{A^{\ast}}x|^{-2}h_l^{\prime}(s)\int_{\partial\Omega}e^{i\left<s^{A^{\ast}}x, \xi\right>}
\\
\left<s^{A^{\ast}}x, n(\xi)\right>
\zeta_0((\xi)_1)d\sigma(\xi)|\det s^A|\,ds.
\end{multline}
We introduce homogeneous coordinates
\begin{align}\label{homcoord}
(s, \alpha)\mapsto\xi(s, \alpha)=s^A(\alpha, \gamma_0(\alpha)).
\end{align}
The Jacobian of the map (\ref{homcoord}) is 
\begin{align*}
\left<s^A(1, \gamma_0^{\prime}(\alpha)), s^{-1}As^A(\alpha, \gamma_0(\alpha))\right>.
\end{align*}
Using homogeneous coordinates, (\ref{int1}) can be written as
\begin{multline}\label{int2}
i\int\zeta_0(\alpha)\int_0^{\infty}|s^{A^{\ast}}x|^{-2}h_l^{\prime}(s)e^{i\left<x, s^A(\alpha, \gamma_0(\alpha))\right>}
\\
\times\left<x, s^A(-\gamma_0^{\prime}(\alpha), 1)(1+(\gamma_0^{\prime}(\alpha))^2)^{-1/2}\right>
\\
\times\left<s^A(1, \gamma_0^{\prime}(\alpha)), s^{-1}As^A(\alpha, \gamma_0(\alpha))\right>|\det s^A|\,ds\,d\alpha.
\end{multline}

Let $\eta:\mathbb{R}\to\mathbb{R}$ be a smooth function supported in $[-\epsilon, \epsilon]$, where
\begin{align}\label{etasupp}
\epsilon=\Theta(\Omega, A)\cdot\min(|\lambda_1|, |\lambda_2|)/(100\cdot 2^{M+2}).
\end{align} 
Then (\ref{int2}) can be written as $\tilde{K}_1(x)+\tilde{K}_2(x)$, where
\begin{multline*}
\tilde{K}_1(x)=
i\int\zeta_0(\alpha)\int_0^{\infty}|s^{A^{\ast}}x|^{-2}h_l^{\prime}(s)e^{i\left<x, s^A(\alpha, \gamma_0(\alpha))\right>}
\\
\times\eta\big(\frac{\left<x, A(\alpha, \gamma_0(\alpha))\right>}{|x|}\big)\left<x, s^A(-\gamma_0^{\prime}(\alpha), 1)(1+(\gamma_0^{\prime}(\alpha))^2)^{-1/2}\right>
\\
\times\left<s^A(1, \gamma_0^{\prime}(\alpha)), s^{-1}As^A(\alpha, \gamma_0(\alpha))\right>|\det s^A|\,ds\,d\alpha,
\end{multline*}
\begin{multline*}
\tilde{K}_2(x)=i\int\zeta_0(\alpha)\int_0^{\infty}|s^{A^{\ast}}x|^{-2}h_l^{\prime}(s)e^{i\left<x, s^A(\alpha, \gamma_0(\alpha))\right>}
\\
\times\big(1-\eta\big(\frac{\left<x, A(\alpha, \gamma_0(\alpha))\right>}{|x|}\big)\big)\left<x, s^A(-\gamma_0^{\prime}(\alpha), 1)(1+(\gamma_0^{\prime}(\alpha))^2)^{-1/2}\right>
\\
\times\left<s^A(1, \gamma_0^{\prime}(\alpha)), s^{-1}As^A(\alpha, \gamma_0(\alpha))\right>|\det s^A|\,ds\,d\alpha.
\end{multline*}

To estimate $\tilde{K}_2(x)$, we integrate by parts with respect to $s$ twice. This yields
\begin{multline*}
\tilde{K}_2(x)=i\int\zeta_0(\alpha)\big(1-\eta\big(\frac{\left<x, A(\alpha, \gamma_0(\alpha))\right>}{|x|}\big)\big)
\\
\int_0^{\infty}g_2(x, s, \alpha)e^{i\left<x, s^A(\alpha, \gamma_0(\alpha))\right>}\,ds\,d\alpha,
\end{multline*}
where
\begin{multline*}
g_2(x, s, \alpha)=\frac{d}{ds}\bigg(\left<x, s^{-1}As^A(\alpha, \gamma_0(\alpha))\right>^{-1}\frac{d}{ds}\bigg(\left<x, s^{-1}As^A(\alpha, \gamma_0(\alpha))\right>^{-1}
\\\times|s^{A^{\ast}}x|^{-2}h_l^{\prime}(s)\left<x, s^A(-\gamma_0^{\prime}(\alpha), 1)(1+(\gamma_0^{\prime}(\alpha))^2)^{-1/2}\right>
\\
\times\left<s^A(1, \gamma_0^{\prime}(\alpha)), s^{-1}As^A(\alpha, \gamma_0(\alpha))\right>|\det s^A|\bigg)\bigg).
\end{multline*}

Note that if $0<\delta<C$ for a sufficiently small constant $C>0$, then for $s$ in the support of $h_l(s)$ and for $x$ in the the support of $1-\eta\big(\frac{\left<x, A(\alpha, \gamma_0(\alpha))\right>}{|x|}\big)$, we have $\left<x, s^{-1}s^AA(\alpha, \gamma_0(\alpha))\right>\ge |x|\cdot\epsilon/2$. Thus 
\begin{align*}
|g_2(x, s, \alpha)|\lesssim |x|^{-3}|h_l^{\prime\prime}(s)|.
\end{align*}
This implies that
\begin{align}\label{K2est}
|\tilde{K}_2(x)|\lesssim|x|^{-3}\int\zeta_0(\alpha)\int |h^{\prime\prime}_l(s)|\,ds\,d\alpha\lesssim 2^{l}|x|^{-3}.
\end{align}
To estimate $\tilde{K}_1(x)$, we integrate by parts with respect to $\alpha$ once and then with respect to $s$ twice, which yields
\begin{align*}
\tilde{K}_1(x)=\int_0^{\infty}\int
g_1(x, s, \alpha)e^{i\left<x, s^{A}(\alpha, \gamma_0(\alpha))\right>}\,d\alpha\,ds,
\end{align*}
where
\begin{multline*}
g_1(x, s, \alpha)=-\frac{d}{ds}\bigg(\left<x, s^{-1}As^A(\alpha, \gamma_0(\alpha))\right>^{-1}\frac{d}{ds}\bigg(\left<x, s^{-1}As^A(\alpha, \gamma_0(\alpha))\right>^{-1}
\\
\times\frac{d}{d\alpha}\bigg(\left<x, s^A(1, \gamma_0^{\prime}(\alpha))\right>^{-1}\zeta_0(\alpha)\eta\big(\frac{\left<x, A(\alpha, \gamma_0(\alpha))\right>}{|x|}\big)
\\
\times i\left<x, s^A(-\gamma_0^{\prime}(\alpha), 1)(1+(\gamma_0^{\prime}(\alpha))^2)^{-1/2}\right>
\\
\times\left<s^A(1, \gamma_0^{\prime}(\alpha)), s^{-1}As^A(\alpha, \gamma_0(\alpha))\right>|s^{A^{\ast}}x|^{-2}h_l^{\prime}(s)|\det s^A|\bigg)\bigg)\bigg).
\end{multline*}
By the choice of $\epsilon$, we have that for $s$ in the support of $h_l(s)$ and for $x$ in the the support of $\eta\big(\frac{\left<x, A(\alpha, \gamma_0(\alpha))\right>}{|x|}\big)$, if $\theta$ denotes the angle between $x$ and $A(\alpha, \gamma_0(\alpha))$, then $\cos(\theta)\le \Omega(\Theta, A)/100$. Since $A(\alpha, \gamma_0(\alpha))$ is tangent to the orbit $\{s^A(\alpha, \gamma_0(\alpha)):\, s>0\}$ at $(\alpha, \gamma_0(\alpha))$, if $0<\delta<C$ for a sufficiently small constant $C$, we have 
\begin{align*}
|\left<x, s^A(1, \gamma_0^{\prime}(\alpha))\right>|^{-1}\ge(\Theta(\Omega, A)/2^{M+2})\cdot|x|.
\end{align*}
It follows that
\begin{align*}
|g_1(x, s, \alpha)|\lesssim|x|^{-2}|\left<x, s^{-1}As^A(\alpha, \gamma_0(\alpha))\right>|^{-2}|h_l^{\prime\prime}(s)|\zeta_0(\alpha)(1+|\gamma^{\prime\prime}(\alpha)|),
\end{align*}
and hence
\begin{multline}\label{K1est}
|\tilde{K}_1(x)|\lesssim\int\int|x|^{-2}|1+|\left<x, s^{-1}As^A(\alpha, \gamma_0(\alpha))\right>||^{-2}
\\
\times|h_l^{\prime\prime}(s)|\zeta_0(\alpha)|(1+|\gamma^{\prime\prime}(\alpha)|)\,d\alpha\,ds.
\end{multline}
It follows from (\ref{K2est}) that
\begin{align}\label{rest3}
\int_{\mathcal{A}_k}|\tilde{K}_2(x)|
\lesssim2^{l-k},
\end{align}
and it follows from (\ref{K1est}) that
\begin{align}\label{rest4}
\int_{\mathcal{A}_k}|\tilde{K}_1(x)|
\lesssim2^{l-k},
\end{align}
and (\ref{rest3}) and (\ref{rest4}) imply (\ref{intann}).
\newline
\indent
Now, (\ref{K2est}) implies that for $0<\delta<C$ we have that $|\tilde{K}_2\cdot\chi_{|t^{A}\cdot|\ge 2^{5M\cdot l}}|$ is bounded above by a radial, decreasing function with $L^1$ norm $\lesssim 1$. It follows that there is a sequence $\{a_n\}$ with $a_n\ge 0$ and $\sum_{n=0}^{\infty}a_n\lesssim 1$ such that for $0<\delta<C$,
\begin{multline}\label{k2maxest}
\Norm{\sup_{t\in(0, \infty)}|(\chi_{|t^{A}\cdot|\ge 2^{5M\cdot l}}\cdot \det(t^A)\tilde{K}_2(t^A\cdot))\ast f(x)|}_{L^2(\mathbb{R}^2)}
\\\lesssim\Norm{\sum_{n=0}^{\infty}a_nM_{2^{-n}, 1}f}\lesssim\delta^{-\epsilon}\Norm{f}_{L^2(\mathbb{R}^2)},
\end{multline}
where we have applied Proposition \ref{maxprop}.
\newline
\indent
We now prove a similar estimate for $\tilde{K}_2$. Observe that (\ref{K1est}) implies that if $0<\delta<C$,
\begin{multline*}
\sup_{t\in(0, \infty)}|(\chi_{|t^{A}\cdot|\ge 2^{5M\cdot l}}\cdot \det(t^A)\tilde{K}_1(t^A\cdot))\ast f(x)|\lesssim
\\
\int\int \sum_{n=0}^{\infty}2^{-n/4}M_{2^{n/4}, 2^{n/4}}f(x)|h^{\prime\prime}_l(x)|\xi_0(\alpha)|(1+|\gamma^{\prime\prime}(\alpha)|\,d\alpha\,ds,
\end{multline*}
and hence by Proposition \ref{maxprop},
\begin{align}\label{k1maxest}
\Norm{\sup_{t\in(0, \infty)}|(\chi_{|t^{A}\cdot|\ge 2^{5M\cdot l}}\cdot \det(t^A)\tilde{K}_1(t^A\cdot))\ast f(x)|}_{L^2(\mathbb{R}^2)}\lesssim\delta^{-\epsilon}\Norm{f}_{L^2(\mathbb{R}^2)}.
\end{align}
Together (\ref{k2maxest}) and (\ref{k1maxest}) prove the result with $\overline{M}$ replaced by $M_1$.
\newline
\indent
It remains to prove the result with $\overline{M}$ replaced by $M_2$. Observe that 
\begin{multline*}
\sigma_j(\xi)=\phi_0(2^{-2M})\alpha_j(\xi)(1-\alpha_j(\xi-2^{M+10}|I_j|(1, \gamma_0^{\prime}(i_j))))
\\\times(1-\alpha_{j+1}(\xi))(\alpha_{j+1}(\xi+2^{M+10}|I_j|(1, \gamma_0^{\prime}(b_{j+1}))))\cdot\tilde{m}_j(\xi),
\end{multline*}
where 
\begin{align*}
\tilde{m}_j(\xi)=\Psi_0(\xi)\psi_0(\rho(\xi))\nu_j(((1/\rho(\xi))^A\xi)_1),
\end{align*}
for some $C^{\infty}$ function $\nu_j$ supported in an interval $I_j^{\ast}$ of width $10|I_j|$ satisfying 
\begin{align*}
|D^i\nu_j|\lesssim |I_j|^{-i}
\end{align*}
for every integer $i\ge 0$. The kernel of the multiplier $\tilde{m}_j$ can be easily written as an integral in homogeneous coordinates. If we can prove that for every annulus $\mathcal{A}_k$ with $k\ge 0$,
\begin{align}\label{nest1}
\int_{\mathcal{A}_k}|\mathcal{F}^{-1}[\tilde{m}_j(\cdot)](x)|\,dx\lesssim l,
\end{align}
then it would follow that the desired result reduces to proving the result of the proposition with $\overline{M}f(x)$ replaced by
\begin{align}\label{rando}
\sup_{t\in(0, \infty)}|(\chi_{|t^A\cdot|\le 2^{20M\cdot l}}\cdot\mathcal{F}^{-1}[\tilde{m}_j(t^{-A}\cdot)])\ast f(x)|.
\end{align}
We now proceed to prove (\ref{nest1}). As before, let $\eta$ be smooth and supported in $[-\epsilon, \epsilon]$, where $\epsilon$ is given by (\ref{etasupp}). Also, as before let $\phi\in C^{\infty}([-1, 1])$ be nonnegative and identically $1$ on $[-1/2, 1/2]$, and for $n\in\mathbb{Z}$ set $\phi_n(\cdot)=\phi(2^{-n-1}\cdot)-\phi(2^{-n}\cdot)$. Define
\begin{align}\label{Phi0}
\Phi_0(x, s, \alpha)=\phi_0(|I_j|\left<x, s^A(1, \gamma_0^{\prime}(\alpha))\right>)\eta\big(\frac{\left<x, s^A(1, \gamma_0^{\prime}(\alpha))\right>}{|x|}\big))
\end{align}
\begin{multline}\label{Phin}
\Phi_n(x, s, \alpha)=(\phi_0(2^{-n-1}|I_j|\left<x, s^A(1, \gamma_0^{\prime}(\alpha))\right>)
\\
-\phi_0(2^{-n}|I_j|\left<x, s^A(1, \gamma_0^{\prime}(\alpha))\right>))\eta\big(\frac{\left<x, s^A(1, \gamma_0^{\prime}(\alpha))\right>}{|x|}\big).
\end{multline}
We decompose the kernel as
\begin{align}\label{nest2}
\mathcal{F}^{-1}[\tilde{m}_j(\cdot)](x)=\frac{1}{(2\pi)^2}[\tilde{K}_j(x)+\sum_{n\ge 0}K_{j, n}(x)],
\end{align}
where
\begin{align*}
K_{j, n}(x)=\int\nu_j(\alpha)\int h_l(s)\Phi_n(x, s, \alpha)e^{i\left<x, s^A(\alpha, \gamma_0(\alpha))\right>}
\\
\times\left<s^A(1, \gamma_0^{\prime}(\alpha)), s^{-1}As^A(\alpha, \gamma_0(\alpha))\right>\,ds\,d\alpha
\end{align*}
and
\begin{align*}
\tilde{K}_j(x)=\int\nu_j(\alpha)\,\big(1-\eta\big(\frac{\left<x, s^A(1, \gamma_0^{\prime}(\alpha))\right>}{|x|}\big)\big)\int h_l(s)e^{i\left<x, s^A(\alpha, \gamma_0(\alpha))\right>}
\\
\times\left<s^A(1, \gamma_0^{\prime}(\alpha)), s^{-1}As^A(\alpha, \gamma_0(\alpha))\right>\,ds\,d\alpha.
\end{align*}
Note that the sum in (\ref{nest2}) has only $\lesssim\log(1+|I_j||x|)$ terms, since $K_{j, n}(x)=0$ if $2^{n-10}|I_j|^{-1}\ge\epsilon|x|$. In particular if $x\in\mathcal{A}_k\cap\text{sup}(K_{j, n})$ then $2^n<<2^k|I_j|$.
\\
\indent For $K_{j, 0}$, we simply estimate $\int_{\mathbb{R}^2}|K_{j, 0}(x)|\,dx$. For a given $(\alpha, s)$, we introduce coordinates 
\begin{align}\label{coord}
(u_1, u_2)\mapsto\xi(u_1, u_2)=u_1s^A(1, \gamma_0^{\prime}(\alpha))+u_2s^{-1}As^A(1, \gamma_0^{\prime}(\alpha)).
\end{align}
The Jacobian of the map (\ref{coord}) is $\approx1$.
Integrating by parts three times in $s$ yields
\begin{multline}\label{pointkj0}
|K_{j, 0}(x)|\lesssim
\\
\int_{s:\,|s-1|\approx 2^{-l}}\int_{\substack{\alpha\in I_j^{\ast}\\\left<x, s^A(1, \gamma_0^{\prime}(\alpha))\right>\le(2|I_j|)^{-1}}}(1+2^{-l}|\left<x, s^{-1}As^A(\alpha, \gamma_0(\alpha)\right>|)^{-3}\,ds\,d\alpha,
\end{multline}
and thus using the change of coordinates (\ref{coord})
\begin{multline}\label{kj0est}
\int_{\mathbb{R}^2}|K_{j, 0}(x)|\,dx\lesssim\\
\int_{I_j^{\ast}}\int\int_{|u_1|\le (2|I_j|)^{-1}}2^{-l}(1+2^{-l}|u_2|)^{-3}\,du_1\,du_2\,d\alpha\lesssim1.
\end{multline}
For $n>0$, we integrate by parts with respect to $\alpha$ once and then with respect to $s$ twice, which yields
\begin{align*}
K_{j, n}(x)=\int\int h_l(s)g_n(x, s, \alpha)e^{i\left<x, s^A(\alpha, \gamma_0(\alpha))\right>}\,d\alpha\,ds,
\end{align*}
where
\begin{multline*}
g_n(x, s, \alpha)=-\frac{d}{ds}\bigg(\frac{1}{\left<x, s^{-1}As^A(\alpha, \gamma_0(\alpha))\right>}\frac{d}{ds}\bigg(\frac{1}{\left<x, s^{-1}As^A(\alpha, \gamma_0(\alpha))\right>}
\\
\frac{d}{d\alpha}\bigg(\frac{1}{\left<x, s^A(1, \gamma_0^{\prime}(\alpha))\right>}\nu_j(\alpha)\Phi_n(x, s, \alpha)
\left<s^A(1, \gamma_0^{\prime}(\alpha)), s^{-1}As^A(\alpha, \gamma_0(\alpha))\right>\bigg).
\end{multline*}
On the support of $h_l(s)$ we have
\begin{align*}
|g_n(x, s, \alpha)|\lesssim
\frac{(1+|x|2^{-n}|I_j|)|\gamma_0^{\prime\prime}(\alpha)|+|I_j|^{-1}}{2^{-2l}|\left<x, s^{-1}As^A(1, \gamma_0^{\prime}(\alpha))\right>|^2|\left<x, s^A(1, \gamma_0^{\prime}(\alpha))\right>|},
\end{align*}
and so
\begin{multline}\label{kjnpoint}
|K_{j, n}(x)|\lesssim\int_{s:\,|s-1|\approx 2^{-l}}\int_{\substack{\alpha\in I_j^{\ast}\\|\left<x, s^A(1, \gamma_0^{\prime}(\alpha))\right>|\\\approx 2^n|I_j|^{-1}}}
\\
\frac{(1+|x|2^{-n}|I_j|)|\gamma_0^{\prime\prime}(\alpha)|+|I_j|^{-1}}{|\left<x, s^A(1, \gamma_0^{\prime}(\alpha)\right>|}
\frac{1}{(1+2^{-l}|\left<x, s^{-1}As^A(1, \gamma_0^{\prime}(\alpha))\right>|)^2}\,d\alpha\,ds.
\end{multline}
Using the change of coordinates (\ref{coord}), it follows that
\begin{multline*}
\int_{\mathcal{A}_k}|K_{j, n}(x)|\,dx\lesssim\int_{s:\,|s-1|\approx 2^{-l}}2^l\int_{\alpha\in I_j^{\ast}}((1+2^{k-n}|I_j|)|\gamma_0^{\prime\prime}(\alpha)|+|I_j|^{-1})
\\\times\int_{\substack{u_1\approx 2^n|I_j|^{-1}\\|u|\approx 2^k}}|u_1|^{-1}\frac{2^{-l}}{(1+2^{-l}|u_1|)^2}\,du\,d\alpha
\\
\lesssim\int_{I_j^{\ast}}(|\gamma_0^{\prime\prime}(\alpha)|+2^{k-n}|I_j||\gamma_0^{\prime\prime}(\alpha)|+|I_j|^{-1})\,d\alpha.
\end{multline*}
By (\ref{left}) we have $\int_{I_j^{\ast}}|I_j||\gamma_0^{\prime\prime}(\alpha)|\,d\alpha\le 2^{-l}$, and so
\begin{align*}
\int_{\mathcal{A}_k}|K_{j, n}(x)|\,dx\lesssim\min\{2^{k-l}, 2^{l-k}\}(2^{k-l-n}+1).
\end{align*}
Since $K_{j, n}$ is identically $0$ on $\mathcal{A}_k$ if $n\ge k$, summing in $n$ and also using (\ref{kj0est}) yields
\begin{align}\label{nint}
\sum_{n\ge 0}\int_{\mathcal{A}_k}|K_{j, n}(x)|\,dx\lesssim k2^{-|k-l|}.
\end{align}
Now we estimate $\tilde{K}_j$. Integrating by parts once in $\alpha$ and then once in $s$ yields
\begin{align}\label{tildpoint}
|\tilde{K}_j(x)|\lesssim\int_{s:\,|s-1|\approx 2^{-l}}\int_{I_j^{\ast}}\frac{(|I_j|^{-1}+|\gamma_0^{\prime\prime}(\alpha)|)}{|x|(1+2^{-l}|\left<x, s^{-1}As^A(\alpha, \gamma_0(\alpha)\right>|)}\,d\alpha\,ds,
\end{align}
and so using the change of coordinates (\ref{coord}) we get
\begin{align}\label{tildint}
\int_{\mathcal{A}_k}|\tilde{K}_j(x)|\,dx\lesssim 1.
\end{align}
Combining (\ref{nint}) and (\ref{tildint}) gives (\ref{nest1}). We now proceed to examine
\begin{align*}\sup_{t\in(0, \infty)}|(\chi_{|t^A\cdot|\le 2^{20M\cdot l}}\cdot\mathcal{F}^{-1}[\tilde{m}_j(t^{-A}\cdot)])\ast f(x)|.
\end{align*}
By (\ref{pointkj0}), for $0<\delta<C$ we have
\begin{multline*}
\sup_{t\in(0, \infty)}|(\chi_{|t^A\cdot|\le 2^{20M\cdot l}}\cdot K_{j, 0}(t^A\cdot))\ast f(x)|\lesssim
\\
\int_{s:\,|s-1|\approx 2{-l}}2^l\int_{\alpha\in |I_j|^{\ast}}|I_j|^{-1}\sum_{n=0}^{C_M\cdot l}M_{|I_j|^{-1}, 2^{l+n/3}|I_j|}f(x)\,d\alpha\,ds,
\end{multline*}
and hence by Proposition \ref{maxprop},
\begin{align}\label{finalkj0}
\Norm{\sup_{t\in(0, \infty)}|(\chi_{|t^A\cdot|\le 2^{20M\cdot l}}\cdot K_{j, 0}(t^A\cdot))\ast f(x)|}_{L^2(\mathbb{R}^2)}\lesssim_{\epsilon}\delta^{-\epsilon}\Norm{f}_{L^2(\mathbb{R}^2)}.
\end{align}
Similarly examining (\ref{kjnpoint}) and (\ref{tildpoint}) leads to
\begin{align}\label{finalkjn}
\Norm{\sup_{t\in(0, \infty)}|(\chi_{|t^A\cdot|\le 2^{20M\cdot l}}\cdot K_{j, n}(t^A\cdot))\ast f(x)|}_{L^2(\mathbb{R}^2)}\lesssim_{\epsilon}\delta^{-\epsilon}\Norm{f}_{L^2(\mathbb{R}^2)}
\end{align}
for $n>0$ and
\begin{align}\label{finaltild}
\Norm{\sup_{t\in(0, \infty)}|(\chi_{|t^A\cdot|\le 2^{20M\cdot l}}\cdot \tilde{K}_{j}(t^A\cdot))\ast f(x)|}_{L^2(\mathbb{R}^2)}\lesssim_{\epsilon}\delta^{-\epsilon}\Norm{f}_{L^2(\mathbb{R}^2)}.
\end{align}
Combining (\ref{finalkj0}), (\ref{finalkjn}) and (\ref{finaltild}) proves the result with $\overline{M}f(x)$ replaced by (\ref{rando}), and the proof of the proposition is complete.
\end{proof}

Finally, we note that the proof of Proposition \ref{bigmaxprop} implies the following $L^1$ kernel estimate.
\begin{proposition}\label{kerestprop}
There exists a constant $C=C(M, \text{Re}(\lambda_1), \text{Re}(\lambda_2), \Theta(\Omega, A))$ such that for $0<\delta<C$, for every $\epsilon>0$ and for every quadruple $(i, j, m, n)$,
\begin{align*}
\Norm{\sup_{t\approx 2^n}|\psi_t\ast K_{i, j, m, n}|}_{L^1(\mathbb{R}^2)}\lesssim_{\epsilon, M, \text{Re}(\lambda_1), \text{Re}(\lambda_2), \Theta(\Omega, A)}1.
\end{align*}
\end{proposition}
The above estimate without the supremum follows immediately from the proof of Proposition \ref{bigmaxprop}. We then simply note that all $L^1$ kernel estimates in the proof of Proposition \ref{bigmaxprop} follow from pointwise estimates, which still hold uniformly in $t$ when the kernel is convolved with $\psi_t$.

\section{Littlewood-Paley Inequalities}

%For each $j, m\equiv 0 \mod N$ where $N$ is a sufficiently large integer to be chosen later , we will choose a rectangle $R_{j, m}$ such that $B_{j, m}\subset R_{j, m}$ and $R_{j_1, m_1}\cap R_{j_2, m_2}=\emptyset$ for $j_1, j_2, m_1, m_2\equiv 0 \mod N$, $(j_1, m_1)\ne (j_2, m_2)$. For each $j$, we let $R_{j, 0}$ be the smallest rectangle containing $B_{j, 0}$ with one side tangent to the level set $\{\rho=1\}$ at the left endpoint of the $j^{\text{th}}$ piece, and let $R_{j, m}$ be the appropriate anisotropic dilate of $R_{j, 0}$. We claim that for $N$ sufficiently large, the $R_{j, m}$ will be pairwise disjoint.  Define $P_{j, m}f=(\chi_{R_{j, m}}\hat{f})^{\check{}}.$ Since the sides of the $R_{j, m}$'s are parallel to at most $O(1/\delta^2)$ different directions, it follows by the result of Karagulyan and Lacey that for any $\epsilon>0$, there exists a constant $C_{\epsilon}$ such that

%$$\Norm{\bigg(\sum_j|P_jf|^2\bigg)^{1/2}}_4\le C_{\epsilon}(1/\delta)^{\epsilon}\Norm{f}_4,$$
%where the sum is taken over all $(j, m)$ with $j, m\equiv 0\mod N$. 

The goal of this section is to prove the following proposition, which is an analog of Proposition $4$ from \cite{car2}. As noted in the introduction, the presence of nonisotropic dilations requires a more complicated application of square function estimates than those used in \cite{car2}, where Proposition $4$ is proved by iteratively applying square function estimates with respect to Fourier projections to parallel strips in $\mathbb{R}^2$.

\begin{proposition}\label{lwpprop}

\iffalse{There is a collection of parallelograms $\{R_{i, j, m, n}\}$ such that for every quadruple $(i, j, m, n)$,
\begin{enumerate}
\item If $m=n=0$, then $R_{i, j, m, n}$ is a rectangle,
\item $R_{i, j, m, n}=(2^{2n})^A(\frac{1+2\delta}{1-2\delta})^{mA}R_{i, j, 0, 0}$,
\item $\sigma_{i, j, m, n}$ is supported in $R_{i, j, m, n}$,
\iffalse{\item There is a constant $C=C(M, \text{Re}(\lambda_1), \text{Re}(\lambda_2), \Theta(\Omega, A)>0$ such that no point of $\mathbb{R}^2$ is contained in more than $C$ of the $R_{i, j, m, n}$.}\fi
\item There is a constant $C=C(M, \text{Re}(\lambda_1), \text{Re}(\lambda_2), \Theta(\Omega, A))>0$ such that $R_{i, j, 0, 0}$ has sidelengths $C\delta$ and $C|I_j|$, and the side with length $C|I_j|$ is parallel to the tangent line to $\partial\Omega$ at $\mathcal{R}_i^{-1}(i_j, \gamma_i(i_j))$.
\end{enumerate}
Now suppose $\{R_{i, j, m, n}\}$ is a collection of parallelograms satisfying $(1)$-$(4)$. Let $\phi:\mathbb{R}^2\to\mathbb{R}$ be a smooth function identically $1$ on the unit cube centered at the origin and supported in its double dilate. Let $L_{i, j, m, n}$ be an invertible linear transformation taking the unit cube centered at the origin to $R_{i, j, m, n}$, and set $\phi_{i, j, m, n}(\cdot)=\phi(L_{i, j, m, n}^{-1}(\cdot))$.}\fi
Let $\epsilon>0$. There is $C=C(M, \text{Re}(\lambda_1), \text{Re}(\lambda_2), \Theta(\Omega, A), \epsilon)
\\>0$ such that if $0<\delta<C$, then the following holds. Let $\{\sigma_{i, j, m, n}\}$ be the partition of unity constructed in section \ref{kerest} for the given value of $\delta$. There are smooth functions $\{\phi_{i, j, m, n}\}$ such that $\phi_{i, j, m, n}$ is identically $1$ on the support of $\sigma_{i, j, m, n}$ and so that if we define $\tilde{P}_{i, j, m, n}$ to be the convolution operator whose multiplier is $\phi_{i, j, m, n}$, then
\begin{align}\label{lwpest}
\Norm{\bigg(\sum_{i, j, m, n}|\tilde{P}_{i, j, m, n}f|^2\bigg)^{1/2}}_4\lesssim_{\epsilon}\delta^{-\epsilon}\Norm{f}_4.
\end{align}
\end{proposition} 

To prove Proposition \ref{lwpprop}, we will need the following lemma, which was originally due to Carleson. A proof can be found in \cite{lrs} (Lemma $4.4$). We state the lemma in full generality, although we will only need the special case $d=2$.
\begin{lemma}\label{carllemma}
Let $A$ be an invertible linear transformation on $\mathbb{R}^d$ and $A^t$ its transpose. Suppose that $\{m_k\}_{k\in\mathbb{N}}$ are bounded, measurable functions on $\mathbb{R}^d$ with disjoint supports. Let $w$ be a bounded, measurable function on $\mathbb{R}^d$. Then for $s\ge 0$ and $f\in\mathcal{S}(\mathbb{R}^d)$, 
\begin{multline*}
\int\sum_k|\mathcal{F}^{-1}[m_k(A^t\cdot)\hat{f}](x)|^2w(x)\,dx
\\
\le C\sup_k\Norm{m_k}_{L_s^2(\mathbb{R}^d)}^2\int\int\frac{\det(A^{-1})}{(1+|A^{-1}y|^s)^2}|f(x-y)|^2\,dy\,w(x)\,dx.
\end{multline*}
\end{lemma}
We state an immediate corollary of this lemma, which we will apply repeatedly in the proof of Proposition \ref{lwpprop}.

\begin{corollary}\label{lwpcor}
Suppose that $\{m_k\}_{k\in\mathbb{Z}}$ are disjoint translates of a smooth compactly supported function adapted to the unit cube in $\mathbb{R}^2$, with the distance between the supports of the $m_k$ at least $O(1)$. Let $R_{\theta}$ be the matrix of rotation by $\theta$ degrees, and for $n\in\mathbb{Z}$ put $A_{n, \theta}=((2^n)^AR_{\theta}(\begin{smallmatrix}\lambda&0\\ 0&\lambda N\end{smallmatrix}))^t$. Then for any $n$, $\theta$ and for any $s>0$,
\begin{align*}
\int\sum_k|\mathcal{F}^{-1}[m_k(A_{n, \theta}^t\cdot)\hat{f}](x)|^2w(x)\,dx\le C\int|f(x)|^2\mathcal{M}_{\lambda, N}w(x)\,dx,
\end{align*}
where $\mathcal{M}_{\lambda, N}:=\sum_{i=0}^{\infty}2^{-i}M_{2^i\lambda, N}$.
\end{corollary}

\begin{proof}[Proof of Proposition \ref{lwpprop}]
Without loss of generality, we may restrict the sum in (\ref{lwpest}) to $i=0$, and so in what follows we will assume $i=0$ and drop the $i$-index. Also, in what follows we will say a collection $\mathcal{R}$ of subsets of $\mathbb{R}^2$ is \textit{almost disjoint} if there is a constant $C=C(M, \text{Re}(\lambda_1), \text{Re}(\lambda_2), \Theta(\Omega, A))>0$ such that every point of $\mathbb{R}^2$ is contained in at most $C$ elements of $\mathcal{R}$. 
\newline
\indent
The main difficulty here introduced by nonisotropic dilations is that unlike the isotropic case, the orbits $\{t^A\xi:\,t>0\}$ need not be straight lines, and thus for fixed $j$ the supports of the $\sigma_{j, m, 0}$ may only be approximated by rectangles with axes whose directions change as $m$ varies. To deal with this difficulty, we group the supports of the $\sigma_{j, m, 0}$ into nested subcollections each of which can be approximated by rectangles with long axes in a single direction, and iteratively apply Corollary \ref{lwpcor}.
\newline
\indent
Note that since $|\delta|\lesssim|I_j|\lesssim 1$, there are $\lesssim\log(1/\delta)$ dyadic intervals $[2^a, 2^{a+1}]$ with $a\le 0$ and $a\in\mathbb{Z}$ such that $2^a\le |I_j|\le 2^{a+1}$ for some $j$, and so if we let $\mathfrak{J}_a=\{j: |I_j|\in [2^a, 2^{a+1}]\}$, we may restrict the sum in $j$ in (\ref{lwpest}) to $\mathfrak{J}_a$ for a single fixed value of $a$, as long as all our estimates are uniform in $a$. By incurring a factor of $\delta^{-\epsilon}$, we may assume that $2^a\le\delta^{-\epsilon}$.
\newline
\indent
Having fixed $a$, we are now ready to construct for each fixed $j$ our nested subcollections of indices $m$. The idea is that for a fixed $j$ and a fixed $m$, the support of $\sigma_{j, m, 0}$ is essentially a $2^a\times\delta$ rectangle, and the support of $\sigma_{j, m^{\prime}, 0}$ for $m^{\prime}$ for $|m^{\prime}-m|\lesssim 2^{-a}$ is contained in a $2^a\times\delta$ rectangle whose direction differs by at most $\lesssim 2^{-a}\delta$. Thus the supports of the functions $\{\sigma_{j, m^{\prime}, 0}\}_{|m-m^{\prime}|\lesssim 2^{-a}}$ are contained in almost disjoint parallel strips of width $\approx\delta$. For such a collection of rectangles, Corollary \ref{lwpcor} may be applied. The union of such rectangles is essentially a $2^{a}\times 2^{-a}\delta$ rectangle. We now iterate this process, grouping together successive $2^{a}\times 2^{-a}\delta$ rectangles whose direction does not change too much to obtain a rectangle of smaller eccentricity. We continue this process until we obtain a $2^a\times 2^a$ square, and then we may apply Corollary \ref{lwpcor}.
\newline
\indent
The nested subcollection of indices $m$ will be constructed ``backwards" with respect to the process described in the previous paragraph. The number of stages required by the process is $N$, where $N$ is the least integer such that $2^{aN}\le\delta$. For each $1\le k\le N$, we will define a collection of indices $m$ denoted by $\mathfrak{M}_{i_1, \ldots, i_k}$, so that $\mathfrak{M}_{i_1, \ldots, i_{k+1}}\subset\mathfrak{M}_{i_1, \ldots, i_k}$ and so that $\mathfrak{M}_{i_1, \ldots, i_N}$ contains at most one element. For each $(i_1, \ldots, i_k)\in\mathbb{Z}^k$, inductively define 
\begin{align*}
\mathfrak{M}_{i_1}=\{m:\, i_1\floor{2^a\delta^{-1}}\le m<(i_1+1)\floor{2^a\delta^{-1}}\},
\end{align*}
\begin{multline*}
\mathfrak{M}_{i_1,\ldots, i_k}=\mathfrak{M}_{i_1, \ldots, i_{k-1}}\cap\{m:\,\sum_{1\le l\le k}i_l\floor{2^{al}\delta^{-1}}\le m
\\\le \sum_{1\le l\le k}i_l\floor{2^{al}\delta^{-1}}+\floor{2^{ak}\delta^{-1}}\}.
\end{multline*}
Then for every $N$-tuple $(i_1, \ldots, i_N)$, $\mathfrak{M}_{i_1, \ldots, i_N}$ contains at most one element.
\newline
\indent
Now let $C=C(M, \text{Re}(\lambda_1), \text{Re}(\lambda_2), \Theta(\Omega, A)>0$ be sufficiently large. There is a collection $\{Q_{j, i_1}\}$ of almost-disjoint cubes of sidelength $C2^a$ such that if $m\in\mathfrak{M}_{i_1, \ldots, i_N}$ then the support of $\sigma_{j, m, 0}$ is contained in $Q_{j, i_1}$. Since $\sup_{\rho(\xi)\le 8}|\nabla\rho(\xi)|\lesssim1$, there is a constant $C>0$ such that for every $j, i_1$ we may cover $Q_{j, i_1}$ with almost disjoint parallel rectangles $R_{j, i_1, i_2}$ of width $C2^{2a}$ and length $1$ so that for every $i_2$, 
\begin{align*}
\bigcup_{m\in\mathfrak{M}_{i_1, i_2}}\text{supp}(\sigma_{j, m, 0})\subset\bigcup_r(R_{j, i_1, i_2}\cap Q_{j, i_1}).
\end{align*}
Repeating this process, for every $2\le k\le N$ and every $k$-tuple $(i_1, \ldots, i_{k-1})$ we obtain almost disjoint parallel rectangles $R_{j, i_1, \ldots, i_k}$ of width $C2^{ka}$ and length $1$ so that for every $i_k$,
\begin{align*}
\bigcup_{m\in\mathfrak{M}_{i_1, \ldots, i_k}}\text{supp}(\sigma_{j, m, 0})\subset(R_{j, i_1, \ldots, i_k}\cap\ldots\cap R_{j, i_1, i_2}\cap Q_{j, i_1}).
\end{align*}
As noted previously, in the case $k=N$, $\bigcup_{m\in\mathfrak{M}_{i_1, \ldots, i_k}}\text{supp}(\sigma_{j, m, 0})$ contains at most one element.
\newline
\indent
Let $\phi:\mathbb{R}^2\to[0, 1]$ be a smooth function that is identically $1$ on the unit cube centered at the origin and supported in its double dilate. If $R$ is any nonisotropic dilate of a rectangle, let $L_R$ be the affine transformation taking $R$ to the unit cube centered at the origin. It follows that if $m\in\mathfrak{M}_{i_1, \ldots, i_N}$, then
\begin{align*}
\text{supp}({\sigma_{j, m, 0}})\subset\{x: \phi(L_{Q_{j, i_1}}x)\prod_{l=2}^N\phi(L_{R_{j, i_1,\ldots, i_l}}x)=1\},
\end{align*}
and so
\begin{align*}
\text{supp}({\sigma_{j, m, n}})\subset\{x: \phi(L_{Q_{j, i_1}}(2^{-n})^Ax)\prod_{l=2}^N\phi(L_{R_{j, i_1,\ldots, i_l}}(2^{-n})^Ax)=1\}.
\end{align*}
Now for each $j, i_1$ let $\psi_{Q_{j, i_1}}$ be a smooth function supported in $4Q_{j, i_1}$ and identically $1$ on $Q_{j, i_1}$, so that for each $j$,
\begin{align*}
\psi_{Q_{j, i_1}}(x)=1,\qquad x\in\bigcup_{m\in\mathfrak{M}_{i_1}}\text{supp}(\sigma_{j, m, 0}).
\end{align*}
For each $(j, m, n)$, let $(i_1, \ldots, i_N)$ be the unique $N$-tuple such that $m\in\mathfrak{M}_{i_1, \ldots, i_N}$, and let
\begin{align*}
\phi_{j, m, n}=\psi_{Q_{j, i_1, r}}((2^{-n})^Ax)\prod_{l=2}^N\phi(L_{R_{j, i_1, \ldots, i_l}, r}(2^{-n})^Ax)
\end{align*}
Let $\tilde{P}_{j, m, n}$ denote the convolution operator with multiplier $\phi_{j, m, n}$. Let $\phi:\mathbb{R}\to [0, 1]$ be a smooth function supported in $(1/4, 4)$ that is identically $1$ on $[1/2, 2]$, and let $P_n$ denote the convolution operator with multiplier $\phi(2^{-n}\rho(\cdot))$. Given an $N$-tuple of indices $(i_1, \ldots, i_N)$, let $m(i_1, \ldots, i_N)$ denote the unique value of $m$ such that $m\in\mathfrak{M}_{i_1, \ldots, i_N}$, and let $m(i_1, \ldots, i_N)$ be undefined otherwise. Let $S_{j, i_1}$ denote the convolution operator with multiplier
\begin{align*}
\psi_{Q_{j, i_1, r}}((2^{-n})^A\cdot).
\end{align*}
For $2\le k\le N$, let $S_{j, i_1, \ldots, i_k, n}$ denote the convolution operator with multiplier
\begin{align*}
\psi_{Q_{j, i_1}}((2^{-n})^AL_{j, i_1, r}\cdot)\phi(L_{R_{j, i_1, \ldots, i_k}}(2^{-n})^A\cdot).
\end{align*}
Then since each index $m$ is contained in at most one $N$-tuple $(i_1, \ldots, i_N)$, it follows that
\begin{multline*}
\int\sum_{j, m, n}|\tilde{P}_{j, m, n}f(x)|^2\,w(x)\,dx=
\\
\int\sum_n\sum_j\sum_{(i_1, \ldots, i_N)}|S_{j, i_1, \ldots, i_N}(\ldots (S_{j, i_1}(P_nf(x))|^2\,w(x)\,dx.
\end{multline*}
Repeatedly applying Corollary \ref{lwpcor}, we have
\begin{multline}\label{longthing}
\int\sum_{j, m, n}|\tilde{P}_{j, m, n}f(x)|^2\,w(x)\,dx
\\
\lesssim_{\epsilon}
\int\sum_n\sum_j\sum_{i_1, \ldots, i_{N-1}} |S_{j, i_1, \ldots, i_{N-1}}(\cdots (S_{j, i_1}(P_nf(x)))\cdots)|^2\,
\\
\times\mathcal{M}_{1, (2^{Na}\delta^{N\epsilon})^{-1}}w(x)\,dx
\\
\lesssim_{\epsilon}\int\sum_n\sum_j\sum_{i_1} |S_{j, i_1, }(P_nf(x))|^2\,\mathcal{M}_{1, (2^{2a}\delta^{2\epsilon})^{-1}}(\cdots
\\(\mathcal{M}_{1, (2^{Na}\delta^{N\epsilon})^{-1}}w(x))\cdots)\,dx
\\
\lesssim_{\epsilon}
\int\sum_n|P_nf(x)|^2\,\mathcal{M}_{2^{-a}\delta^{-\epsilon}, 1}(\mathcal{M}_{1, (2^{2a}\delta^{2\epsilon})^{-1}}(\cdots
\\(\mathcal{M}_{1, (2^{Na}\delta^{N\epsilon})^{-1}}w(x))\cdots))\,dx
\\
\lesssim_{\epsilon}\delta^{-\epsilon}\Norm{\bigg(\sum_n|P_nf|^2\bigg)^{1/2}}_4^2
\\\times\Norm{\mathcal{M}_{2^{-a}\delta^{0\epsilon}, 1}(\mathcal{M}_{1, (2^{2a}\delta^{\epsilon})^{-1}}(\cdots(\mathcal{M}_{1, (2^{Na}\delta^{N\epsilon})^{-1}}w)\cdots))}_2.
\end{multline}
By Proposition \ref{maxprop}, we have
\begin{multline}\label{longthing2}
\Norm{\mathcal{M}_{2^{-a}\delta^{0\epsilon}, 1}(\mathcal{M}_{1, (2^{2a}\delta^{\epsilon})^{-1}}(\cdots(\mathcal{M}_{1, (2^{Na}\delta^{N\epsilon})^{-1}}w)\cdots))}_2
\\\lesssim_{\epsilon}\delta^{-\epsilon}\Norm{w}_2.
\end{multline}
Since the operator $f\mapsto\bigg(\sum_n|P_nf|^2\bigg)^{1/2}$ corresponds to a vector-valued singular integral on the space of homogeneous type given by nonisotropic balls and Lebesgue measure with all associated constants $\lesssim1$, we have
\begin{align}\label{singint}
\Norm{\bigg(\sum_n|P_nf|^2\bigg)^{1/2}}_4\lesssim\Norm{f}_4.
\end{align}
Combining (\ref{longthing}), (\ref{longthing2}) and (\ref{singint}), we have
\begin{align}\label{lasthing}
\int\sum_{j, m, n}|\tilde{P}_{j, m, n}f(x)|^2\,w(x)\,dx\lesssim_{\epsilon}\delta^{-\epsilon}\Norm{f}_4^2\Norm{w}_2,
\end{align}
and the result follows by duality.
\end{proof}

\section{Proof of the main theorem}
In this section, we combine the ingredients developed in previous sections to prove Proposition \ref{mainprop}. The argument will closely follow \cite{car2}. As noted previously, we only need prove Proposition \ref{mainprop} in the case that $\Omega$ has smooth boundary, with a constant depending only on $M, \text{Re}(\lambda_1), \text{Re}(\lambda_2), \Theta(\Omega, A), \epsilon$.

\begin{proof}[Proof of Propostion \ref{mainprop}]
Let $Sf(x)=\big(\int_0^{\infty}|\psi_t\ast f(x)|^2\frac{dt}{t}\big)^{1/2}$. Let $\mathfrak{S}$ be the nonisotropic sector bounded by the orbits $\{t^A\xi:\,t>0\}$ and $\{t^A\xi^{\prime}:\,t>0\}$, where $\xi=(\xi_1, \xi_2)$ is the unique point in $\partial\Omega$ with $\xi_1=-1/8$ and $\xi_2>0$ and $\xi^{\prime}=(\xi_1^{\prime}, \xi_2^{\prime})$ is the unique point in $\partial\Omega$ with $\xi_1=1/8$ and $\xi_2>0$. Assume without loss of generality that $\hat{f}$ is supported in $\mathfrak{S}$. By incurring a factor of $\log(1/\delta^{N(M, \text{Re}(\lambda_1), \text{Re}(\lambda_2))})$ we may restrict the domain of integration in $t$ to the set 
\begin{align*}
E=\bigcup_{\substack{n\equiv 0\mod\\\log(1/\delta^{N(M, \text{Re}(\lambda_1), \text{Re}(\lambda_2))})}}(2^n, 2^{n+1}],
\end{align*} 
where $N(M, \text{Re}(\lambda_1), \text{Re}(\lambda_2))$ is as in Proposition \ref{imp2}.
Now, if $u, t\in E$ with $u<t$, then either $u, t$ are contained in the same dyadic interval and $u/t>1/2$, or $u, t$ are contained in distinct dyadic intervals, and $u/t<1/\delta^{N(M, \text{Re}(\lambda_1), \text{Re}(\lambda_2))}$. Using Plancherel, we have
\begin{multline*}
\Norm{Sf}_4^4=\int\bigg|\int_0^{\infty}|\psi_t\ast f(x)|^2\frac{dt}{t}\bigg|^2\,dx=
\\
\int\int_0^{\infty}\int_{0}^{\infty}|\psi_t\ast f(x)|^2|\psi_u\ast f(x)|^2\frac{dt}{t}\frac{du}{u}\,dx=
\\\int_0^{\infty}\int_0^{\infty}\int|(\phi(\frac{\rho(\cdot)}{t})\hat{f}(\cdot))\ast(\phi(\frac{\rho(\cdot)}{u})\hat{f}(\cdot))(\xi)|^2\,d\xi\frac{dt}{t}\frac{du}{u}.
\end{multline*}
Restricting the integration in $t$ and $u$ to $E$, we have
\begin{multline*}
\int_E\int_E\int|(\phi(\frac{\rho(\cdot)}{t})\hat{f}(\cdot))\ast(\phi(\frac{\rho(\cdot)}{u})\hat{f}(\cdot))(\xi)|^2\,d\xi\frac{dt}{t}\frac{du}{u}\lesssim
\\
\bigg(\int\int_{1/2<t/u<2}+\int\int_{u/t<\delta^{-N(M, \text{Re}(\lambda_1), \text{Re}(\lambda_2))}}\bigg)
\int|(\phi(\frac{\rho(\cdot)}{t})\hat{f}(\cdot))
\\
\ast(\phi(\frac{\rho(\cdot)}{u})\hat{f}(\cdot))(\xi)|^2\,d\xi.
\end{multline*}
Using Propositions \ref{algdisj} and \ref{imp2}, for every $\epsilon>0$ we can essentially bound this by
\begin{multline*}
\delta^{-\epsilon}\bigg(\int_0^{\infty}\int_0^{\infty}\int\sum_{\substack{j, m, n\\ j^{\prime}, m^{\prime}, n^{\prime}}}|(\sigma_{0, j, m, n}(\cdot)\phi(\frac{\rho(\cdot)}{t})\hat{f}(\cdot))
\\
\ast(\sigma_{0, j^{\prime}, m^{\prime}, n^{\prime}}(\cdot)\phi(\frac{\rho(\cdot)}{u})\hat{f}(\cdot))(\xi)|^2\,d\xi\,\frac{dt}{t}\,\frac{du}{u}
\\
+\int_0^{\infty}\int_0^{\infty}\int\sum_{j^{\prime}, m^{\prime}, n^{\prime}}|(\phi(\frac{\rho(\cdot)}{t})\hat{f}(\cdot))\ast(\sigma_{0, j^{\prime}, m^{\prime}, n^{\prime}}(\cdot)\phi(\frac{\rho(\cdot)}{u})\hat{f}(\cdot))(\xi)|^2\,d\xi\,\frac{dt}{t}\,\frac{du}{u}\bigg).
\end{multline*}
Let 
\begin{align*}
Tf(x)=\bigg(\int_0^{\infty}\sum_{j, m, n}|\mathcal{F}[\sigma_{0, j, m, n}(\cdot)\phi(\frac{\rho(\cdot)}{t})\hat{f}(\cdot)](x)|^2\frac{dt}{t}\bigg)^{1/2}.
\end{align*} 
Then the above implies that
\begin{align*}
\Norm{Sf}_4^4\lesssim_{\epsilon}\delta^{-\epsilon}(\Norm{Tf}_4^4+\Norm{Sf}_4^2\Norm{Tf}_4^2),
\end{align*}
which implies 
\begin{align}\label{st}
\Norm{Sf}_4\lesssim_{\epsilon}\delta^{-\epsilon}\Norm{Tf}_4.
\end{align}
Using Proposition \ref{imp3}, we have
\begin{multline*}
\Norm{Tf}_4=\bigg(\int\bigg|\int_0^{\infty}\sum_{j, m, n}|\mathcal{F}[\sigma_{0, j, m, n}(\cdot)\phi(\frac{\rho(\cdot)}{t})\hat{f}(\cdot)](x)|^2\frac{dt}{t}\bigg|^{2}\,dx\bigg)^{x1/4}
\\
=\bigg(\int\bigg|\sum_{ j, m, n}\int_0^{\infty}|((\psi_t\ast K_{0, j, m, n})\ast(\tilde{P}_{0, j, m, n}f))(x)|^2\frac{dt}{t}\bigg|^{2}\,dx\bigg)^{1/4}
\\
\lesssim\delta^{1/2}\bigg(\int\bigg|\sum_{j, m, n}\sup_{t\approx 2^n}|(\psi_t\ast K_{0, j, m, n})\ast(\tilde{P}_{0, j, m, n}f)(x)|^2\bigg|^{2}\,dx\bigg)^{1/4}
\\
=\delta^{1/2}\Norm{\bigg(\sum_{j, m, n}\sup_{t\approx 2^n}|(\psi_t\ast K_{0, j, m, n})\ast(\tilde{P}_{0, j, m, n}f)(x)|^2\bigg)^{1/2}}_4.
\end{multline*}
Now let $\omega\in\mathcal{S}(\mathbb{R}^2)$ with $\Norm{\omega}_{L^2(\mathbb{R}^2)}=1$. We have
\begin{multline*}
\int\sum_{j, m, n}\sup_{t\approx 2^n}|\psi_t\ast K_{0, j, m, n}\ast \tilde{P}_{0, j, m, n}f(x)|^2\omega(x)\,dx
\\
\lesssim\int\sum_{j, m, n}\Norm{\sup_{t\approx 2^n}|\psi_t\ast K_{0, j, m, n}|}_1|\tilde{P}_{0, j, m, n}f(x)|^2|\sup_{t\approx 2^n}|\psi_t\ast K_{0, j, m, n}\ast w(x)||\,dx
\\
\lesssim_{\epsilon}\delta^{-\epsilon}\int\sum_{j, m, n}|\tilde{P}_{0, j, m, n}f(x)|^2\overline{M}\omega(x)\,dx\\\lesssim_{\epsilon}\Norm{\bigg(\sum_{j, m, n}|\tilde{P}_{0, j, m, n}f(x)|^2\bigg)^{1/2}}_4^2\Norm{\overline{M}w}_2\lesssim_{\epsilon}\delta^{-\epsilon}\Norm{f}_4^2,
\end{multline*}
where in the second inequality we have used Proposition \ref{kerestprop} and in the last inequality we have used Propositions \ref{bigmaxprop} and \ref{lwpprop}. Using (\ref{st}) and taking the supremum over all such weights $\omega$, we have
\begin{align*}
\Norm{Sf}_4\lesssim_{\epsilon}\delta^{-\epsilon}\Norm{Tf}_4\lesssim_{\epsilon}\delta^{1/2-\epsilon}\Norm{f}_4.
\end{align*}
\end{proof}

\begin{proof}[Proof that Proposition \ref{mainprop} implies Theorem \ref{mainthm}]
Let $C$ be as in the statement of Proposition \ref{mainprop}. We will now decompose the Bochner-Riesz multipliers in a standard fashion. Let $\phi_0:\mathbb{R}\to\mathbb{R}$ be a smooth function identically $1$ on $[-1, 1]$ and supported in $[-2, 2]$ so that $\phi_0(|\cdot|)$ is a radial, decreasing function on $\mathbb{R}^2$. It is easy to see that we can find smooth functions $\phi_1:\mathbb{R}\to\mathbb{R}$ and $\phi_2:\mathbb{R}\to\mathbb{R}$ satisfying the following:
\begin{enumerate}
\item For each $k\ge 0$, 
\begin{align*}
|D^k\phi_1(x)|\lesssim_{k}1,
\end{align*}
\begin{align*}
|D^k\phi_2(x)|\lesssim_{k}1,
\end{align*}
\item There is a constant $C^{\prime}>0$ such that $\phi_1$ is supported in $[C^{\prime}, 1]$,
\item We can write
\begin{multline*}
(1-\rho(\xi))^{\lambda}=\phi_0(2^{2M}|\xi|)+(\phi_0(2^{-2M}|\xi|)-\phi_0(2^{2M}|\xi|))\phi_1(\rho(\xi))
\\
+\sum_{k=\lceil{\log(C)}\rceil}^{\infty}2^{-k\lambda}\phi_2(2^k(1-\rho(\xi))).
\end{multline*}
\end{enumerate}
By the triangle inequality,
\begin{multline*}
\Norm{G^{\lambda}f}_{4}\lesssim\Norm{\bigg(\int_0^{\infty}\bigg|\mathcal{F}^{-1}[\phi_0(2^{2M}t^{-1}|\cdot|)]\ast f(x)\bigg|^2\,\frac{dt}{t}\bigg)^{1/2}}_{4}
\\
+\Norm{\bigg(\int_0^{\infty}\bigg|\mathcal{F}^{-1}[(\phi_0(2^{-2M}t^{-1}|\cdot|)-\phi_0(2^{2M}t^{-1}|\cdot|))\phi_1(t^{-1}\rho(\cdot))]\ast f(x)\bigg|^2\,\frac{dt}{t}\bigg)^{1/2}}_{4}
\\
+\sum_{k=1}^{\infty}2^{-k\lambda}\Norm{\bigg(\int_0^{\infty}\bigg|\mathcal{F}^{-1}[\phi_1(2^k(1-t^{-1}\rho(\cdot)))]\ast f(x)\bigg|^2\,\frac{dt}{t}\bigg)^{1/2}}_{4}.
\end{multline*}
The first term is clearly $\lesssim\Norm{f}_4$. By Proposition \ref{mainprop}, the third term is also $\lesssim\Norm{f}_4$ if $\lambda>-1/2$. By vector-valued singular integrals, the second term is bounded by
\begin{align*}
\Norm{\bigg(\int_0^{\infty}\bigg|\mathcal{F}^{-1}[\phi_1(t^{-1}\rho(\cdot))]\ast f(x)\bigg|^2\,\frac{dt}{t}\bigg)^{1/2}}_{4},
\end{align*}
and it is straightforward to adapt the proof of Proposition \ref{mainprop} to show that this is $\lesssim\Norm{f}_4$.
\end{proof}

\end{document}